\documentclass[a4paper,11pt]{scrartcl}

\usepackage{amssymb,amsthm,amsmath, paralist,bbding,pifont}
\usepackage{enumitem}
\usepackage{multicol}
\usepackage{graphicx}
\usepackage{todonotes}
\usepackage[utf8]{inputenc}
\usepackage[normalem]{ulem}

\usepackage{setspace}  
\usepackage{scrpage2}

\usepackage{url}
\usepackage{hyperref}
\usepackage{float}
\usepackage{subfig}
\usepackage{tikz}
\usepackage{pgfplots}
\usetikzlibrary{decorations.text}
\usetikzlibrary{shapes}
\usetikzlibrary{trees}
\usetikzlibrary{calc}
\usepgfplotslibrary{fillbetween}

\usepackage[ruled]{algorithm2e}

\ohead{\pagemark}
\ihead{\headmark}
\ofoot{}

\usepackage[paper=a4paper,left=25mm,right=25mm,top=30mm,bottom=30mm]{geometry}

\newcommand{\R}{\mathbb{R}}
\newcommand{\Oad}{\mathcal{O}^{\text{ad}}}
\newtheorem{theorem}{Theorem}
\newtheorem{definition}[theorem]{Definition}

\newtheorem{lemma}[theorem]{Lemma}
\newtheorem{proposition}[theorem]{Proposition}
\newtheorem{remark}[theorem]{Remark}

\newcommand{\hide}[1]{}

\begin{document}
\title{\LARGE Tracing locally Pareto optimal points by numerical integration}
\author{Matthias Bolten${}^*$, Onur Tanil Doganay${}^*$, Hanno Gottschalk${}^*$\\ and Kathrin Klamroth${}^*$\\[2ex] ${}^*$IMACM, School of Mathematics and Natural Science,\\ University of Wuppertal, D-42119 Wuppertal\\
{\small \texttt{\{bolten,doganay,gottschalk,klamroth\}@math.uni-wuppertal.de}}}

\maketitle
\begin{abstract}
    We suggest a novel approach for the efficient and reliable approximation of the Pareto front of sufficiently smooth unconstrained bi-criteria optimization problems. Optimality conditions  formulated for weighted sum scalarizations of the problem yield a description of (parts of) the Pareto front as a parametric curve, parameterized by the scalarization parameter (i.e., the weight in the weighted sum scalarization). Its sensitivity w.r.t.\ parameter variations can be described by an ordinary differential equation (ODE). Starting from an arbitrary initial Pareto optimal solution, the Pareto front can then be traced by numerical integration. We provide an error analysis based on Lipschitz properties and suggest an explicit Runge-Kutta method for the numerical solution of the ODE. The method is validated on bi-criteria convex quadratic programming problems for which the exact solution is explicitly known, and numerically tested on complex bi-criteria shape optimization problems involving finite element discretizations of the state equation.
\end{abstract}


\hspace{.1cm}


\section{Introduction}
Multi-criteria optimization models gain more and more importance in economical and in technical applications. Decision makers have to balance between economical and ecological criteria and compromise between reliability and cost, to mention only two examples. In this situation, a concise representation of the set of Pareto optimal solutions, i.e., the set of those solutions that can not be improved in one criterion without deterioration in at least one other criterion, provides important trade-off information and thus  supports the decision maker in identifying relevant solution alternatives. In this paper, we aim at the reliable and efficient  approximation of the Pareto front of convex and sufficiently smooth unconstrained bi-criteria optimization problems.

Scalarization methods are a prevalent tool to compute representations and approximations of the Pareto front. We refer to   \cite{ehrg:mult:2005,miet:nonl:1998,ruzi:appr:2005} for a thorough introduction to the field and to \cite{das:aclo:1997,marl:thew:2010} for a discussion of the pros and cons of the weighted sum scalarization. 
Assuming differentiability, optimality conditions like, for example, the classical KKT-conditions, can be used to derive Pareto optimal solutions, see, for example, \cite{ehrg:mult:2005,hill:nonl:2001}. The Pareto front can then be recovered using subdivision techniques \cite{dell:cove:2005,jahn:mult:2006,schu:seto:2013}, sensitivities with respect to the scalarization parameters \cite{Eich09}, 
or continuation and predictor-corrector methods  \cite{Eich09,Mart18,peit:expl:2017,peit:asur:2018,ring:hand:2012,schm:pare:2008,schu:onco:2005}. The latter usually rely on scalarizations, leading to  single-objective counterpart problems that depend on one or several scalarization parameters (e.g., the weights in the case of weighted sum scalarizations) and that can hence be interpreted as parametric optimization problems. Under appropriate differentiability assumptions, predictor-corrector-type methods can then be related to the single-objective case  (see, for example, \cite{allg:intr:2003,gudd:para:1987}).

This paper is organized as follows: In Section~\ref{sec:Paretotracing} the unconstrained bi-criteria optimization problem is introduced  together with the (slightly atypical) notation that is used throughout this paper (Section~\ref{sec:notation}). Under appropriate differentiability assumptions, the problem of tracing the Pareto front is reformulated as an explicit ordinary differential equation (ODE). 
We derive existence and continuity results for its solution (Section~\ref{subsec:ODE}), assuming local Lipschitz continuity of the Hessians of both objective functions. In Section~\ref{subsec:approx} the results are extended to the case that initial Pareto critical solutions can only be approximated. We note that this case is particularly relevant for complex real world applications as discussed in the case study presented in Section~\ref{subsec:shapes}. This representation of the Pareto front is the basis for the application of well-established numerical integration methods for Pareto front tracing. We suggest the application of Runge-Kutta methods and provide local and global error estimates in Section~\ref{sec:numericalintegration}. The numerical results presented in Section~\ref{sec:results} for quadratic test problems (Section~\ref{subsec:QP}) and for complex bi-criteria shape optimization problems (Section~\ref{subsec:shapes}) validate the high solution quality and the efficiency of the approach.

\section{Pareto tracing using ODEs}\label{sec:Paretotracing}

\subsection{Some notation for bi-criteria Pareto optimality}\label{sec:notation}
We first collect some basic notation and facts on bi-criteria unconstrained optimization. For a detailed introduction into the field of multi-criteria optimization, see, for example, the books \cite{ehrg:mult:2005,miet:nonl:1998}. Let $J:\R^n\to \R^2$ be a bi-criteria objective function that is second order differentiable with continuous derivative. For notational reasons that will become clear later, we denote the individual objective functions by $J_0$ and $J_1$, i.e.,  $J=(J_0,J_1)$. For $x,x'\in \R^n$, $x'$ \emph{dominates} $x$, if $J_i(x')\leq J_i(x)$ for $i=0,1$  and $J_i(x')<J_i(x)$ for at least one $i\in\{0,1\}$. A solution $x\in \R^n$ is \emph{Pareto optimal}, if it is not dominated by any solution $x'\in\R^n$ and $x$ is \emph{locally Pareto optimal}, if there exists some open neighborhood $\mathcal{U}$ of $x$ such that no $x'\in\mathcal{U}$ dominates $x$. The corresponding outcome vector $J(x)$ is called \emph{nondominated} or \emph{locally nondominated}, respectively. A search direction $d\in \R^n$ is a bi-criteria descent direction at $x$, if $\nabla J_i^\top(x) d<0$ for $i\in \{0,1\}$, where $\nabla J_i(x)$ is the gradient of $J_i$. If at $x$ there is no bi-criteria descent direction, $x$ is called \emph{Pareto critical}. Obviously, being Pareto critical is a necessary but not sufficient condition for being locally Pareto optimal. For $\varepsilon>0$, $x$ is called $\varepsilon$-Pareto critical, if there is no search direction $d\in\R^n$ such that $\nabla_x J_{i}(x)^\top d\leq -\varepsilon \|d\|$ and $\nabla_x J_{j}(x)^\top d<0$, where $i,j\in\{0,1\}$ with $i\not=j$.

In the following, we work with the \emph{weighted sum scalarization} given by $J_\lambda=(1-\lambda)J_0+\lambda J_1$ for $\lambda\in [0,1]$ the preference parameter. $x$ then is (locally) optimal with respect to $\lambda$, if $J_\lambda(x)\leq J_\lambda(x')$ for $x'\in \R^n$ ($x'\in\mathcal{U}$) and $x$ is critical for $J_\lambda$ if $\nabla_xJ_\lambda(x)=(1-\lambda)\nabla_xJ_0(x)+\lambda \nabla_xJ_1(x) =0$. If this condition holds approximately such that $\|\nabla_xJ_\lambda(x)\|\leq\varepsilon$ for some small $\varepsilon >0$, we say that $x$ is $\varepsilon$-critical with respect to $J_\lambda$. For $\lambda\in(0,1)$ this clearly implies $x$ being $\varepsilon'$-Pareto critical with $\varepsilon'=\frac{\varepsilon}{\min\{\lambda,(1-\lambda)\}}$.

Let $\nabla^2_xJ_i(x)$ be the Hessian matrix of $J_i$ at $x$, $i\in \{0,1\}$ and likewise $\nabla^2_xJ_\lambda(x)$ the Hessian of $J_\lambda$. We say that $x$ fulfills the second order optimality conditions for $J_\lambda$ \emph{strictly} if $x$ is $J_\lambda$-critical and $\nabla^2_xJ_\lambda$ is strictly positive definite. If $x$ fulfills strict second order $J_\lambda$-optimality, it is locally a (unique) $J_\lambda$-optimal point which also implies local Pareto optimality of $x$  \cite{hill:nonl:2001,miet:nonl:1998}.

Note that the parameter values $\lambda=0$ and $\lambda=1$ correspond to the single-criteria minimization of the individual objective functions $J_0$ and $J_1$, respectively. If in this case the optimal solution is (locally) unique, then it is (locally) Pareto optimal. This is, for example, the case when the second order optimality condition is satisfied strictly. Otherwise, the set of optimal solutions may contain (locally) weakly Pareto optimal solutions that are not locally Pareto optimal, where a  solution $x'\in\R^n$ is called (locally) \emph{weakly Pareto optimal} if there is no other solution $x\in\R^n$ ($x\in\mathcal{U}$ for some open neighborhood $\mathcal{U}$ of $x'$) such that $J_i(x)<J_i(x')$ for $i=0,1$.

Conversely, if the outcome set $\{J(x)\,:\, x\in\R^n\}$ is $\R^2_+$-convex, i.e., if the set $\{J(x)+r\,:\, x\in\R^n,\, r\in\R^2 \text{~and~} r_i\geq 0,\, i=0,1\}$ is convex, then all Pareto optimal solutions can be retrieved by minimizing $J_\lambda$ with an appropriate scalarization parameter $\lambda\in[0,1]$. Moreover, if the outcome set is $\R^2_+$-convex and -compact, then the nondominated set is connected in the outcome space. We refer again to \cite{ehrg:mult:2005} and the references therein for a more detailed discussion of this and of related topics.

\subsection{Implicit and explicit ODEs for local Pareto optimality}\label{subsec:ODE}
From now on we assume that $\nabla^2_xJ_i$ is locally Lipschitz continuous  with Lipschitz constant $L_{H}(x,\delta)$ on the ball $B_\delta(x)$ of radius $\delta$ centered at $x$, i.e. $\|\nabla^2_xJ_i(x)-\nabla^2_xJ_i(x')\|\leq L_H(x,\delta)\|x-x'\|$, $i=0,1$, where $\|\cdot\|$ is the spectral norm. 

Let us first assume that we have found points $x(\lambda)$, which fulfill criticality for $J_\lambda$ on some interval $\lambda\in[\lambda_l,\lambda_u]\subseteq[0,1]$. Suppose furthermore that $x(\lambda)$ is a differentiable function of $\lambda$. Differentiating the first order optimality conditions $\nabla_xJ_\lambda(x(\lambda))=0$ with respect to $\lambda$, we obtain
\begin{equation}\label{eqn:implicit_bicriteria_ode}
    \nabla_x^2 J_\lambda( x (\lambda))\, \Dot{ x }(\lambda) = \nabla_x J_0( x (\lambda)) - \nabla_x J_1( x (\lambda)).
\end{equation}
If $x(\lambda)$ furthermore fulfills the second order optimality conditions with respect to $J_\lambda$ strictly, this implicit ordinary differential equation (ODE) can be rearranged to a standard ODE $\dot x(\lambda)=f(\lambda,x(\lambda))$ with $f$ defined by
\begin{equation}
     \Dot{ x }(\lambda) = \left(\nabla_x^2 J_\lambda( x (\lambda))\right)^{-1}\left(\nabla_x J_0( x (\lambda)) - \nabla_x J_1( x (\lambda))\right)=f(\lambda,x(\lambda)).
     \label{eqn:bicriteria_ode}
\end{equation}
Let us conversely assume that we have found $x_0$ which fulfills the strict second order optimality condition for $J_{\lambda_0}$. We easily see that the right hand side of \eqref{eqn:bicriteria_ode} as a function of $x$ is Lipschitz on some open neighborhood $\mathcal{U}$ of $x_0$ with a Lipschitz constant $L_{f}$ that is uniform in $\lambda$ on some interval $[\lambda_l,\lambda_u]\subseteq[0,1]$: 

\begin{lemma}
\label{lem:f_is_Lipshitz}
Let $\lambda\in[0,1]$ and let $\Lambda(\lambda,x)$ be the smallest eigenvalue of $\nabla^2_xJ_{\lambda}(x)$.  Then 
\begin{itemize}
    \item[(i)] $\Lambda(\lambda,x)$ is locally Lipschitz in $x$ with Lipschitz constant $L_{H}(x,\delta)$ on  $B_{\delta}(x)$;
    \item[(ii)] $\Lambda(\lambda,x)$ is Lipschitz in $\lambda$ on $[0,1]$ with Lipschitz constant $L_\lambda(x)=\|\nabla_x^2J_0(x)\|+\|\nabla_x^2J_1(x)\|$;
    \item[(iii)] Let $1>\varrho>0$, then for $x'\in  B_{\delta}(x)$ and $\lambda'\in[0,1]$ such that
    \[
    L_{H}(x,\delta)\|x-x'\|+L_\lambda(x)|\lambda-\lambda'|\leq (1-\varrho)\Lambda(\lambda,x),
    \]
    we have $\Lambda(\lambda',x')\geq\varrho \Lambda(\lambda,x)$;
    \item[(iv)] Let $[\lambda_{l},\lambda_u]$ containing $\lambda$ and $B_{\delta}(x)$ be given such that  $L_{H}(x,\delta)\delta+L_\lambda(x)\max\{\lambda_u-\lambda,\lambda-\lambda_l\}\leq (1-\varrho)\Lambda(\lambda,x)$. This is always possible as $L_{H}(x,\delta)$ is monotonically decreasing in $\delta$. Then $f(\lambda',x')$  is uniformly (in $\lambda '$) Lipschitz in $x'$ on $ [\lambda_l,\lambda_u]\times B_{\delta}(x)$ with Lipschitz constant bounded by 
    \[
    L_f(x,\delta ,\varrho)=2\left(\left(\frac{1}{\varrho\,\Lambda(\lambda,x)}\right)C_2(x,\delta )+\left(\frac{1}{\varrho\,\Lambda(\lambda,x)}\right)^2L_{H}(x,\delta )C_1(x,\delta )\right),
    \]
    where $C_1=\max_{i\in\{0,1\}}\sup_{x'\in B_{\delta }(x)}\|\nabla_x J_i(x')\|$ and $C_2=\max_{i\in\{0,1\}}\sup_{x'\in B_{\delta }(x)}\|\nabla_x^2 J_i(x')\|$. 
\end{itemize}
\begin{proof}
(i) Let $x'$ and $x''$ be in $B_\delta(x)$. Without loss of generality we assume that $\Lambda(\lambda,x')\geq \Lambda(\lambda,x'')$. Then, 
\begin{align*}
0&<\Lambda(\lambda,x')-\Lambda(\lambda,x'')= \inf_{u\in\R^n:\|u\|=1}u^\top\nabla^2_xJ_\lambda(x')u-\Lambda(\lambda,x'')\\
&\leq \inf_{u\in\R^n:\|u\|=1}u^\top\nabla^2_xJ_\lambda(x'')u+\sup_{u\in\R^n:\|u\|=1}u^\top\left(\nabla^2_xJ_\lambda(x')-\nabla^2_xJ_\lambda(x'')\right)u-\Lambda(\lambda,x'')\\
&= \sup_{u\in\R^n:\|u\|=1}u^\top\left(\nabla^2_xJ_\lambda(x')-\nabla^2_xJ_\lambda(x'')\right)u=\|\nabla^2_xJ_\lambda(x')-\nabla^2_xJ_\lambda(x'')\|\\
&\leq L_{H}(x,\delta)\, \|x'-x''\|.
\end{align*}
(ii) We proceed similarly as in (i) and obtain for $\lambda',\lambda''\in[0,1]$ with $\Lambda(\lambda',x)\geq \Lambda(\lambda'',x)$
\begin{align*}
0&<\Lambda(\lambda',x)-\Lambda(\lambda'',x)
\leq \sup_{u\in\R^n:\|u\|=1}u^\top\left(\nabla^2_xJ_{\lambda'}(x)-\nabla^2_xJ_{\lambda''}(x)\right)u\\
&= \|\nabla^2_xJ_{\lambda'}(x)-\nabla^2_xJ_{\lambda''}(x)\|\leq \left(\|\nabla^2_xJ_0(x)\|+\|\nabla^2_xJ_1(x)\|\right) |\lambda'-\lambda''|.
\end{align*}
(iii) For $x'\in B_\delta(x)$, (iii) now follows from (i) and (ii) by 
\begin{align*}
    \Lambda(\lambda',x')&=\Lambda(\lambda,x)+\left(\Lambda(\lambda',x)-\Lambda(\lambda,x)\right)+\left(\Lambda(\lambda',x')-\Lambda(\lambda',x)\right)\\
    &\geq \Lambda(\lambda,x)-L_{H}(x,\delta)\|x-x'\|-L_\lambda(x)|\lambda-\lambda'|\geq \varrho \Lambda(\lambda,x).
\end{align*}
(iv) We first recall the following fact: Let $A_0$, $A_1$ be two strictly positive definite $n\times n$ matrices with lowest eigenvalue not smaller than $\varepsilon$ and let $A_\tau=\tau A_1+(1-\tau)A_0$ for $\tau\in[0,1]$. Then $A_\tau$ is again positive definite with lowest eigenvalue not smaller than $\varepsilon$ and we obtain by using the sub-multiplicativity of the spectral norm 
\begin{align*}
\|A^{-1}_{1}-A^{-1}_{0}\| &= \left\|\int_0^1 \frac{d}{d\tau}A^{-1}_\tau \, \text{d}\tau\right\|=\left\|\int_0^1 A_\tau^{-1}\frac{d}{d\tau}A_{\tau} A_\tau^{-1} \, \text{d}\tau\right\|\\
&=\left\|\int_0^1 A_\tau^{-1}(A_1-A_0) A_\tau^{-1} \, \text{d}\tau\right\| \leq \int_0^1 \left\|A_\tau^{-1}(A_1-A_0) A_\tau^{-1}\right\| \, \text{d}\tau \leq \frac{1}{\varepsilon^2} \| A_1-A_0\|.
\end{align*}
 Consequently, for $x',x''\in B_{\delta }(x)$ and $\lambda'\in[\lambda_l,\lambda_u]$ we can use (iii) and the above estimate with $\varepsilon=\rho\Lambda(\lambda,x)$, see (iii), and obtain
\begin{align*}
    \|f(\lambda',x')-f(\lambda',x'')\| &\leq \left\|\left(\nabla^2_xJ_{\lambda'}(x')\right)^{-1}\left( \nabla_x J_0(x')-\nabla_x J_0(x'')-\nabla_x J_1(x')+\nabla_x J_1(x'')\right)\right\| \\
    &+ \left\|\left(\left(\nabla^2_xJ_{\lambda'}(x')\right)^{-1}-\left(\nabla^2_xJ_{\lambda'}(x'')\right)^{-1}\right)\left( \nabla_x J_0(x'')-\nabla_x J_1(x'')\right)\right\|\\
    &\leq \frac{1}{\varrho\Lambda(\lambda,x)}\left(\sup_{x'''\in B_{\delta }(x)}\| \nabla_x^2 J_0(x''')\|+\sup_{x'''\in B_{\delta }(x)}\|\nabla_x^2 J_1(x''')\|\right)\|x'-x''\|\\
    &+ \left(\frac{1}{\varrho\Lambda(\lambda,x)}\right)^2 \left\|\nabla^2_xJ_{\lambda'}(x')-\nabla^2_xJ_{\lambda'}(x'')\right\|\sup_{x'''\in B_{\delta }(x)}\left\|\nabla_x J_0(x''')-\nabla_xJ_1(x''')\right\|\\
    &\leq2\left(\frac{1}{\varrho\Lambda(\lambda,x)}C_2(x,\delta)+\left(\frac{1}{\varrho\Lambda(\lambda,x)}\right)^2L_{H}(x,\delta)C_1(x,\delta)\right)\|x'-x''\| .
\end{align*}
\end{proof}
\end{lemma}

The locally Lipschitz property of $f(\lambda,x)$ established in Lemma \ref{lem:f_is_Lipshitz} (iv) provides us with the crucial input to the Picard-Lindel\"of theorem, that can now be applied as follows:

\begin{theorem}\label{theo:Picard_Lindelof}
Let $x_0\in \R^n$ and $\lambda_0\in[0,1]$ such that $x_0$ fulfills the strict second order optimality conditions with respect to $J_{\lambda_0}$.  Using the notation of Lemma~\ref{lem:f_is_Lipshitz}, let $\Delta,\delta>0$ and $1>\varrho>0$ such that $L_{H}(x_{0},\delta)\delta+L_\lambda(x_0)\Delta\leq(1-\varrho)\Lambda(\lambda_0,x_{0})$.   Let  $C_3(x_{0},\delta,\Delta)=\sup_{\lambda\in[\lambda_0-\Delta,\lambda_0+\Delta]\atop x\in B_\delta(x_{0})}\| f(\lambda,x)\|$.  Let furthermore $\Delta'=\min\{\Delta,\delta/C_3(x_{0},\delta,\Delta)\}$ and $\lambda_l=\lambda_0-\Delta'$, $\lambda_u=\lambda_0-\Delta'$.  Then,
\begin{itemize}
    \item[(i)] The solution $x(\lambda)$ of \eqref{eqn:bicriteria_ode} with initial condition $x(\lambda_0)=x_0$ at $\lambda_0$ exists and is unique locally on the interval $[\lambda_l,\lambda_u]$. $x(\lambda)$ is continuously differentiable on this interval;
    \item[(ii)] $x(\lambda)$ can be extended to a solution of \eqref{eqn:bicriteria_ode} to a maximal  time interval $(\lambda_l',\lambda_u')\subseteq [0,1]$ containing $\lambda_0$ such that $\Lambda(\lambda,x(\lambda))>0$ for $\lambda \in (\lambda_l',\lambda_u')$ and either $\lambda_u=1$ ($\lambda_l=0$) or $\Lambda(\lambda,x(\lambda))$ has accumulation point $0$ as $\lambda \nearrow\lambda_u'$ ($\lambda \searrow\lambda_l'$);
    \item[(iii)] On this interval, $x(\lambda)$ fulfills the strict second order optimality conditions with respect to $J_\lambda$ and thus is locally $J_\lambda$ optimal and locally Pareto optimal with respect to $J$.
\end{itemize}
\end{theorem}
\begin{proof}
(i) By Lemma \ref{lem:f_is_Lipshitz} (iv), the conditions of the Picard Lindel\"of theorem are fulfilled. The assertion thus follows from the local existence and uniqueness of the solutions of ODEs, see e.g.\ \cite[Theorem 15.2]{MR2439721}.  

Statement (ii) immediately follows from the fact that, since $\Lambda(\lambda_u,x(\lambda_u))\geq \rho \Lambda(\lambda_0,x(\lambda_0))$,  the argument from (i) can be iterated with $\lambda_0$ replaced by $\lambda_{u}=\lambda_0^{(1)}$ and $x_0$ with $x(\lambda_{u})$. This procedure can be iterated, until $\lambda_0^{(n)}$ is either reaching one or $\Lambda(\lambda_0^{(n)},x(\lambda_0^{(n)}))$ is approaching $0$. Now define the maximal upper boundary $\lambda_u'=\lim_{n\to\infty}\lambda_0^{(n)}$. An analogous argument holds for the minimal lower bound $\lambda_l'$.

To see (iii), we recall that for any $\lambda\in(\lambda_l',\lambda_u')$, \eqref{eqn:bicriteria_ode} implies \eqref{eqn:implicit_bicriteria_ode} and thus \begin{equation}\label{eqn:integrate_criticality}
\nabla_xJ_\lambda(x(\lambda))=\nabla_xJ_{\lambda_0}(x_0)+\int_{\lambda_0}^\lambda\frac{d}{d\tau} \nabla_xJ_\tau(x(\tau))\, \text{d}\tau=0,
\end{equation}
hence $x(\lambda)$ is critical for $J_\lambda$ and thus Pareto critical. As furthermore $\lambda\in(\lambda_l',\lambda_u')$,  $\nabla^2_xJ_{\lambda}(x(\lambda))$ is strictly positive definite as by (ii) $\Lambda(\lambda,x)>0$ holds, hence $x(\lambda)$ fulfills strict second order optimality for $J_\lambda$ and is locally Pareto optimal.
\end{proof}

\begin{remark}
(i) For notational convenience we consider an unrestricted domain, i.e., all  $x\in\R^n$ are feasible solutions. However, all results of Section~\ref{sec:Paretotracing} easily generalize to the case, where $J_i(x)$ is only defined on an open subset of $\R^n$. In this case, in the local constructions given e.g.\ in Lemma~\ref{lem:f_is_Lipshitz} and Theorem~\ref{theo:Picard_Lindelof} the constants $\delta>0$ have to be chosen smaller than the distance to the boundary of the domain of definition and the results still hold with the obvious adaptations on the maximal intervals of existence $(\lambda_l',\lambda_u')$. 

(ii) In the degenerate case that the two objective functions are equal, i.e., if $J_0=J_1$, then $J_\lambda=J_0=J_1$ for all $\lambda\in[0,1]$. In this case, any optimal solution of $J_0$ (assuming that it exists) is Pareto optimal, and the nondominated set consists of exactly one (ideal) outcome vector. Then 
\eqref{eqn:bicriteria_ode}  becomes  $\dot{x}(\lambda)=0$, which is in accordance with the fact that the nondominated set consists of a unique outcome vector.
\end{remark}

Several estimates for the numerical approximation of $x(\lambda)$ rely on the regularity of $x(\lambda)$. We therefore recall the following standard result on the regularity of solutions to ODEs. 
\begin{lemma}
\label{lem:regularity}
Assume that $J_i$, $i\in\{0,1\}$, is $p+2$ times differentiable with locally bounded $p+2$-nd derivative, $p\in\mathbb{N}_0$. Let $[\lambda_l,\lambda_u]\subset(\lambda_l',\lambda_u')$ be a closed interval in the maximal interval from Theorem~\ref{theo:Picard_Lindelof}(ii).
Then, $x(\lambda)$ is $p+1$ times differentiable with bounded $p+1$st derivative on $[\lambda_l,\lambda_u]$.
\end{lemma}

\begin{proof}
To prove  the $p+1$-order regularity, we note that, if a matrix $A$  is invertible, the operation of inverting is $C^\infty$ on a neighborhood of $A$. Denoting the right hand side of \eqref{eqn:bicriteria_ode}  with $f(\lambda,x)=f^{(0)}(\lambda,x)$, we see that $f^{(0)}$ is $p$ times differentiable in $\lambda$ and $x$. For $l=1,\ldots,p$, we recursively define  $f^{(l)}(\lambda,x)=\frac{\partial}{\partial\lambda}f^{(l-1)}(\lambda,x)+\nabla_xf^{(l-1)}(\lambda,x)^\top f(\lambda,x)$ where $f^{(l)}$ is $p-l$ times differentiable in $x$ and $\lambda$ and locally bounded where $p=l$. Differentiating $x^{(l)}(\lambda)=\left(\frac{d}{d\lambda}\right)^lx(\lambda)=f^{(l-1)}(\lambda,x(\lambda))$ with respect to $\lambda$, we see that $x^{(l+1)}(\lambda)=f^{(l)}(\lambda,x(\lambda))$ for $l=0,\ldots,p$ exists and is bounded on $[\lambda_l,\lambda_u]$ if $l=p$.
\end{proof}

\subsection{Approximately Pareto critical initial conditions and numerical stability}\label{subsec:approx}
In Theorem~\ref{theo:Picard_Lindelof} we assumed that the initial value $x_0$ fulfills the strict second order optimality conditions for $J_{\lambda_0}$. Thus $x_0$ is the (local) optimum to the single-criteria optimization problem given by the objective function $J_{\lambda_0}$. In many applications, we do not know $x_0$, but have to use approximate solutions to the optimization problem posed by $J_{\lambda_0}$, instead.    Suppose that $x_{0,k}\to x_0$ are the iterates of some optimization algorithm started sufficiently close to $x_0$ such that the optimization problem is convex and convergence is guaranteed. For example, $x_{0,k}$ can be obtained by a gradient descent method or Newton-type method applied to  $J_{\lambda_0}$.

Ultimately, we may assume that $x_{0,k}$ is sufficiently close to $x_0$ such that also $\nabla^2_xJ_{\lambda_0}(x_{0,k})$ is strictly positive definite. Assuming that the optimization algorithm that produces $x_{0,k}$ applies a gradient based stopping criterion, e.g. $\|\nabla_x J_{\lambda_0}(x_{0,k})\|\leq \varepsilon$ for some $\varepsilon>0$, the terminal output $x_{0,k}$ is $\varepsilon$-$J_{\lambda_0}$ critical. In the following we see that starting the ODE \eqref{eqn:bicriteria_ode} in $x_{0,k}$ provides $\varepsilon$- critical solutions $x_k(\lambda)$ with respect to $J_\lambda$ for $\lambda$ in some interval containing $\lambda_0$.

In many applications, the function $f(\lambda,x(\lambda))$  has to be approximated using numerical schemes $f_l(\lambda,x)$ with limited accuracy. Here $l$ is some parameter that controls the numerical error in the sense that $\varepsilon_l(\mathcal{C},\lambda)=\sup_{x\in \mathcal{C}}\|f(\lambda,x)-f_l(\lambda,x)\|\to 0$ if $l\to \infty$ and $\mathcal{C}\subseteq\R^n$ is compact. It is therefore desirable to be able to control the effect of the numerical error in $f-f_l$ on the solution of \eqref{eqn:bicriteria_ode} along with the error caused by the error in the initial condition $x_0-x_{0,k}$. The following proposition uses the standard repertoire of ODE theory to provide comprehensive estimates:  

\begin{proposition}
\label{prop:eps-crit_front}
Let $x_0$ fulfill the strict second order optimality condition with respect to $J_{\lambda_0}$ and $x_{0,k}\to x_{0}$ as $k\to\infty$. Then,
\begin{itemize}
    \item[(i)] Let $\varepsilon>0$. For  $k$  sufficiently large, solutions $x_{k}(\lambda)$ to \eqref{eqn:bicriteria_ode} started with initial condition $x_k(\lambda_0)=x_{0,k}$ at $\lambda_0$ exist on some maximal intervals $(\lambda_{l,k}',\lambda_{u,k}')\subset[0,1]$ and $x_k(\lambda)$ is $J_\lambda$ $\varepsilon$-critical and hence  $\varepsilon'=\frac{\varepsilon}{\min\{\lambda,(1-\lambda)\}}$-Pareto critical for $\lambda\in (\lambda_{l,k}',\lambda_{u,k}')$;
    \item[(ii)]  Let $I=[\lambda_l,\lambda_u]\subset(\lambda_l',\lambda_u')$ a compact sub-interval of the maximal interval from (i) where $x(\lambda)$ is shown to exist. 
    Let $\Lambda(I)=\inf_{\lambda\in I}\Lambda(\lambda,x(\lambda))$, $L_{H}(I,\delta)=\sup_{\lambda\in I}L_{H
    }(x(\lambda),\delta)$ and $C_i(\delta,I)=\sup_{\lambda\in I} C_i(x(\lambda),\delta)$, $i=\{1,2\}$. Let furthermore $0<\delta  <\Lambda(I)/L_{H}(I,\delta)$, which is always possible as $L_{H}(I,\delta)$ is finite and monotonically increasing in $\delta$. We also set
    \[
    L_f(\delta ,I)=2\left(\left(\frac{1}{\Lambda(I)-\delta L_{H}(I,\delta)}\right)C_2(\delta,I )+\left(\frac{1}{\Lambda(I)-\delta L_{H}(I,\delta)}\right)^2L_{H}(I,\delta )C_1(\delta,I )\right).
    \]
    Then, for $k$ sufficiently large, $x_k(\lambda)$ exists for $\lambda \in [\lambda_l,\lambda_u] $ and
    \begin{equation} \label{eqn:uniform_convergence_initial}
    \|x(\cdot)-x_k(\cdot)\|_{C(I,\R^n)}\leq \|x_0-x_{0,k}\| \, e^{L_f(\delta,I)\max\{\lambda_0-\lambda_l,\lambda_u-\lambda_0\}},
\end{equation}
where $\|\cdot\|_{C(I,\R^n)}$ stands for the maximum norm on $I$. Hence, $x_k(\lambda)$ converges with the same rate to the locally $J_\lambda$ and locally Pareto optimal point $x(\lambda)$ as $x_{0,k}$ converges to the locally $J_{\lambda_0}$ optimal and locally Pareto optimal point $x_0$.
\item[(iii)] Let, in addition, $f_l$ be a locally Lipschitz function such that $f-f_l\to 0$ uniformly on compact sets. Then, for $\delta$ as in (ii)  and  $k,l$ sufficiently large, the solution $x_{k;l}(\lambda)$ of $\dot x_{k;l}(\lambda)=f_l(\lambda,x_{k;l}(\lambda))$ with initial condition  $x_{k;l}(\lambda_0)=x_{0,k}$ at $\lambda_0$ exists on $I$ and we have the estimate  
\begin{align} \label{eqn:uniform_convergence_inititial_and_f}
\begin{split}
    \|x(\cdot)-x_{k;l}(\cdot)\|_{C(I,\R^n)}&\leq \|x_0-x_{0,k}\| \, e^{L_f(\delta,I)\max\{\lambda_0-\lambda_l,\lambda_u-\lambda_0\}}\\
    &+\frac{1}{L_f(\delta,I)}\left(e^{L_f(\delta,I)\max\{\lambda_0-\lambda_l,\lambda_u-\lambda_0\}}-1\right)\|f_l-f\|_{C(\overline{\mathcal{U}(I,\delta)},\R^n)},
    \end{split}
\end{align}
where $\mathcal{U}(I,\delta)=\bigcup_{\lambda\in I}B_\delta(x(\lambda))$ and $\|\cdot\|_{C(\overline{\mathcal{U}(I,\delta)},\R^n)}$ is the maximum norm on $\overline{\mathcal{U}(I,\delta)}$.
\end{itemize}
   
\end{proposition}
\begin{proof}
(i) If $k$ is sufficiently large such that $\delta=\|x_{0,k}-x_0\|$ fulfills $\delta L_{H}(x,\delta)<\Lambda(\lambda_0,x)$, $\Lambda(\lambda_0,x_{0,k})>0$ and thus $x_k(\lambda)$ exists for some maximal interval $(\lambda_{l,k}',\lambda_{u,k}')$ by repeating the proof of Theorem \ref{theo:Picard_Lindelof} (ii). Furthermore, as $J_{\lambda_0}(x)$ is continuous in $x$,   $x_{0,k}$ is $\varepsilon$-critical with respect to $J_{\lambda_0}$ if $k$ is sufficiently large. By integration as in \eqref{eqn:integrate_criticality} one furthermore obtains that 
\[
\nabla_xJ_\lambda (x_k(\lambda))=\nabla_xJ_{\lambda_0} (x_{0,k}).
\]
Thus, $x_k(\lambda)$ then is $\varepsilon$-critical for $J_\lambda$ for $\lambda\in I$ if $k$ is sufficiently large. The statement on $\varepsilon'$-Pareto criticality then follows as in Subsection \ref{sec:notation}.

(ii) Let now $[\lambda_l,\lambda_u]\subseteq (\lambda_l',\lambda_u')$ be some closed interval and let $\delta>0$ be sufficiently small such that $0<\delta  <\Lambda(I)/L_{H}(I,\delta)$. Let $k$ be sufficiently large such that $\|x_0-x_{0,k}\|< \delta e^{-L_f(\delta,I)\max\{\lambda_0-\lambda_l,\lambda_u-\lambda_0\}}$. Then, this in particular implies $x_{0,k}\in B_\delta(x_0)\subseteq \mathcal{U}(I,\delta)=\bigcup_{\lambda\in I}B_\delta(x(\lambda))$.  By application of Lemma~\ref{lem:f_is_Lipshitz}(iv) with $\varrho=1-\delta L_{H}(I,\delta)/\Lambda(I)$, $L_{H}(I,\delta)$ gives an upper bound for the uniform Lipschitz constant of $f$ on  $\mathcal{U}(I,\delta)$. It follows that $x_k(\lambda)$ exists on some interval $I_n=[\lambda_{l,k},\lambda_{u.k}]\subseteq I$ containing $\lambda_0$. Application of the Gronwall lemma leads to the well known estimate on the continuous dependence on the initial condition (see e.g. Theorem 12.1 in \cite{MR2439721})  
\begin{equation}\label{eqn:Gronwall}
   \|x(\lambda)-x_k(\lambda)\| \leq \|x_0-x_{0,k}\|\,  e^{L_f(I,\delta)|\lambda-\lambda_0|}\leq \|x_0-x_{0,k}\|\,  e^{L_f(I,\delta)\max\{\lambda_0-\lambda_l,\lambda_u-\lambda_0\}}<\delta, 
\end{equation}
for $\lambda\in I_n$. Therefore,  $x_k(\lambda)\in \mathcal{U}(I,\delta)$ and $x_k(\cdot)$ can be further extended beyond $I_n$. The above estimate applied repeatedly shows that $I=[\lambda_l,\lambda_u]$  is contained in the maximal interval of existence $I'_n=(\lambda_{l,n}',\lambda_{u,n}')$ for $x_k(\cdot)$ since for $\lambda\in I$, $x_k(\lambda)$ never leaves $\mathcal{U}(I,\delta)$ and the inequality \eqref{eqn:Gronwall} is valid for all $\lambda\in I$, which proves the proposition's second assertion.

(iii) Statement (iii) is also covered by Theorem~12.1 in \cite{MR2439721} by essentially the same arguments as in (ii). We leave the details to the reader.
\end{proof}

If the computational complexity is known for the computation of $x_{0,k}$ and $f_l$ with a given precision $\varepsilon$,  inequality \eqref{eqn:uniform_convergence_inititial_and_f}  provides the basis for finding an efficient balance between the cost of approximating the initial condition $ x_{0}$ and approximating the function $f$.

\section{Pareto front tracing by numerical integration }\label{sec:numericalintegration}
For an approximation $x_k$ of the Pareto front we have to solve the ODE
\begin{equation}
\Dot{ x_k }(\lambda) = f(\lambda,x_k(\lambda))
     \label{eqn:bicriteria_ode_for_xk}
\end{equation}
with right hand side $f$ as defined in \eqref{eqn:bicriteria_ode}. To approximate \eqref{eqn:bicriteria_ode_for_xk} by numerical integration we need an initial value $x_{0,k}=x(\lambda_0)$ at $\lambda_0$. An obvious choice is $\lambda_0 = 0$. We could also start at $\lambda_0 = 1$, following the ($\varepsilon$-)Pareto critical points backwards, or at a compromise solution obtained, for example, for $\lambda_0=0.5$. Note that the problem might not be well-defined for $\lambda_0 = 0$ or $\lambda_0 = 1$, in which case another starting point $\lambda_0$ is chosen. This allows for following the Pareto critical points in two directions at the same time, solving two independent initial value problems separately. In the following, we denote the $i$th iterate of a numerical integration method by $x_{i,k}$. 

The most basic method for solving \eqref{eqn:bicriteria_ode_for_xk} numerically is the explicit Euler method. It approximates the derivative of $x_k$ by 
\[
    \Dot{x_k}(\lambda) \approx \frac{ x_k(\lambda+h) -  x_k(\lambda)}{h},
\]
yielding
\[
     x_k(\lambda + h) \approx  x_{1,k} := x_k(\lambda) + h \left(\nabla_x^2 J_\lambda( x_k (\lambda))\right)^{-1}\left(\nabla_x J_0( x_(\lambda)) - \nabla_x J_1( x_k (\lambda))\right),
\]
where $h>0$ denotes the step-size of the method.
The global error of the Euler method behaves like $C h$, with a constant $C$ depending on the problem~\cite{MR3559553}.

When higher order is demanded, Runge-Kutta methods \cite{MR1638643,MR1510879} can be used.
\begin{definition}[Explicit Runge-Kutta method]\label{def:erk}
Let $s \in \mathbb{N}$, $h > 0$ and let $a_{2,1}, a_{3,1}, a_{3,2}, \dots, a_{s,1},$ $a_{s,2}, \dots, a_{s,s-1}, b_1, \dots, b_s, c_2, \dots, c_s \in \mathbb{R}$. Then the method
\begin{align}
    k_1 & = f(\lambda_0,x_{0,k}) \nonumber\\
    k_2 & = f(\lambda_0 + c_2 h, x_{0,k} + h a_{2,1} k_1) \nonumber\\
    k_3 & = f(\lambda_0 + c_3 h, x_{0,k} + h (a_{3,1} k_1 + a_{3,2} k_2)) \nonumber\\
    & \vdots \label{eqn:erk}\\
    k_{s} & = f(\lambda_0 + c_s h, x_{0,k} + h (a_{s,1} k_1 + \dots + a_{s,s-1} k_{s-1})) \nonumber\\
    x_{1,k} & = x_{0,k} + h (b_1 k_1 + \dots + b_s k_s) \nonumber
\end{align}
is called \emph{s-stage explicit Runge-Kutta method} for \eqref{eqn:bicriteria_ode_for_xk}.
\end{definition}
\begin{definition}[c.f.~\cite{MR1227985}, Definition II.1.2, p.~134]
A Runge-Kutta method  \eqref{eqn:erk} is of \emph{order $p$} if for sufficiently smooth problems \eqref{eqn:implicit_bicriteria_ode} there exists a $K$ independent from $h$ such that 
\[
    \|x_k(\lambda_0+h) - x_{1,k}\| \leq K h^{p+1}.
\]
\end{definition}
The following rigorous error estimate holds true:
\begin{theorem}[c.f.~\cite{MR1227985}, Theorem II.3.1, p.~157]\label{theo:order}
If the Runge-Kutta method \eqref{eqn:erk} is of order $p$ and if $f(\lambda,x_k(\lambda))$ is $p$-times continuously differentiable, then we have
\[
\|x_k(\lambda_0 + h) - x_{1,k}\| \leq h^{p+1} \left( \frac{1}{(p+1)!} \max\limits_{t \in [0,1]} \|x_k^{(p+1)}(\lambda_0 + th)\| + \frac{1}{p!} \sum\limits_{i=1}^s |b_i| \max\limits_{t \in [0,1]} \|k_i^{(p)}(th)\| \right).\]
\end{theorem}
So, for the present $f$ in order for an order $p$ Runge-Kutta method to be applicable it has to be continuously differentiable $p$ times. This is the case, if the objective functions $J_0$ and $J_1$ are $(p+2)$ times continuously differentiable.
Theorem~\ref{theo:order} holds for each step of the Runge-Kutta method, so for $j=1,\dots,N$ and using $x_{j-1,k}$ as initial value in step $j$ we have estimates
\begin{equation}\label{eqn:rk_local_error_estimates}
    \|e_j\| := \|x_k(\lambda_0 + j h) - x_{j,k}\| \leq C h^{p+1}.
\end{equation}
Using similar ideas as those that were used to prove Proposition~\ref{prop:eps-crit_front} and using the fact that due to Lemma~\ref{lem:f_is_Lipshitz} $f$ is Lipschitz in $\lambda$ we can show the following:
\begin{theorem}[c.f.~\cite{MR1227985}, Theorem II.3.4, p.~160]
Let $\mathcal{U}$ be a neighborhood of $\{(\lambda,x_k(\lambda))| \lambda \in I\}$ where $x_k(\lambda)$ is the exact solution of \eqref{eqn:bicriteria_ode_for_xk} and $I$ as defined in Section~\ref{sec:Paretotracing}. Suppose that in $\mathcal{U}$
\[
\left\|\frac{\partial f}{\partial x}\right\| \leq L_{f_x}
\]
and that the local estimates \eqref{eqn:rk_local_error_estimates} hold in $\mathcal{U}$. Then the global error
\[
E = x_k(\lambda_u) - x_{N,k}
\]
can be estimated as
\[
\|E\| \leq h^{p} \frac{C}{L_{f_x}} \left(e^{L_{f_x}|I|}-1\right),
\]
given that $h$ is small enough to remain in $\mathcal{U}$.
\end{theorem}

In our numerical experiments we are using 2nd-order and 4th-order Runge-Kutta methods.

The simple 2nd-order method is given by
\[
    x_{1,k} = x_{0,k} + h f(0+\frac{h}{2},x_{0,k} + \frac{h}{2} f(0,x_{0,k})).
\]
The 4th-order accurate classical Runge-Kutta method (or RK4-method) is given by
\begin{align*}
    k_1 & = f(\lambda_0,x_{0,k})\\
    k_2 & = f(\lambda_0 + \frac{h}{2},x_{0,k}+\frac{k_1}{2}) \\
    k_3 & = f(\lambda_0 + \frac{h}{2},x_{0,k}+\frac{k_2}{2}) \\
    k_4 & = f(\lambda_0 + h,x_{0,k} + k_3) \\
    x_{1,k} & = x_{0,k} + h\left(\frac{1}{6} k_1 + \frac{2}{6} k_2 + \frac{2}{6} k_3 + \frac{1}{6} k_4\right),
\end{align*}
so it is obtained by choosing:
\[
\begin{array}{c|cccc}
0 & & & & \\
c_2 = \frac{1}{2} & a_{2,1} = \frac{1}{2} & & & \\
c_3 = \frac{1}{2} & a_{3,1} = 0 & a_{3,2} = \frac{1}{2} & & \\
c_4 = 1 & a_{4,1} = 0 & a_{4,2} = 0 & a_{4,3} = 1 & \\
\hline
& b_1 = \frac{1}{6} & b_2 = \frac{1}{3} & b_3 = \frac{1}{3} & b_4 = \frac{1}{6} 
\end{array}
\]
This representation is known as \emph{Butcher tableau} \cite{MR1227985}. RK4 is of order $p = 4$.

To obtain the Pareto front numerically using an arbitrary explicit Runge-Kutta scheme for a given starting point $\lambda_0$, we use Algorithm~\ref{alg:pareto_front_tracing}.

\begin{algorithm}
\DontPrintSemicolon
\SetKwInOut{Input}{Input}
\SetKwInOut{Output}{Output}
\Input{start point $(\lambda_0,x_k(\lambda_0))$, number of integration points $N$, number of steps $s$ and parameters $a_{i,\ell}, b_i, c_i$, $i=1,\dots,s$, $\ell=1,\dots,i-1$ of the chosen explicit Runge-Kutta method}
\Output{approximations to points on the Pareto front $(\lambda_j,x_k(\lambda_j))$, $j=1,\dots,N$}
h = $(\lambda_u - \lambda_0)/N$\;
\For{$j = 1,\dots,N$}{
    \For{$i = 1,\dots,s$}{
        $k_i = f(\lambda_0 + jh + \sum\limits_{\ell=2}^i c_\ell h, x_{j-1,k} + h \sum\limits_{\ell=1}^{i-1} a_{i,\ell} k_i )$\;
    }
    $x_{j,k} = x_{j-1,k} + h \sum\limits_{\ell=1}^s b_\ell k_\ell$\;
}
\caption{Pareto front tracing}
\label{alg:pareto_front_tracing}
\end{algorithm}
\begin{remark}
Analogously to traditional time integration, we have used only integration forward in $\lambda$, here. The Pareto front can also be traced backwards by going from $\lambda_0$ up to $\lambda_l$ and reverting the $\lambda$-direction using proper scaling.
\end{remark}

\section{Numerical results}\label{sec:results}

The approach presented in the previous sections is tested on a simple bi-criteria convex quadratic optimization problem (Section~\ref{subsec:QP}) as well as on a case study in bi-criteria shape optimization (Section~\ref{subsec:shapes}).

\subsection{Pareto tracing for bi-criteria convex quadratic optimization}\label{subsec:QP}

We first consider an unconstrained and strictly convex bi-criteria optimization problem with two quadratic objective functions $J_i:\R^n\rightarrow\R$, $i=0,1$ given by
$J_i(x)=\frac{1}{2}(x-\chi_i)^TQ_i(x-\chi_i)$ with 
strictly positive definite matrices $Q_i\in\R^{n\times n}$ and arbitrary but fixed vectors $\chi_i\in\R^n$. 
Since this problem allows for an analytic description of the Pareto optimal set (see, for example, \cite{tour:onbi:2019}), 
it is particularly well-suited to evaluate the quality of approximated Pareto fronts. 

Indeed, due to the strict convexity of $J_i$, $i=0,1$, every Pareto optimal solution must be the unique optimal solution $x(\lambda)$ of a weighted sum scalarization $J_\lambda=(1-\lambda)J_0+\lambda J_1$ with $\lambda\in[0,1]$, satisfying the first-order optimality condition $\nabla_x J_\lambda(x(\lambda))=0$. Conversely, since the Hessian $\nabla_x^2 J_\lambda(x(\lambda))=(1-\lambda)Q_0+\lambda Q_1$ is positive definite for all $\lambda\in[0,1]$ (irrespective of $x(\lambda)$), every such solution $x(\lambda)$ satisfies the second order optimality condition strictly and is thus Pareto optimal. 
The first order optimality condition yields an explicit formula for the solution $x(\lambda)$ as a function of $\lambda\in[0,1]$:
\begin{align}
& \nabla_x J_{\lambda}(x(\lambda)) = (1-\lambda)Q_0(x(\lambda)-\chi_0) + \lambda Q_1(x(\lambda)-\chi_1) = 0 \notag \\
\Leftrightarrow \quad & x(\lambda) = [(1-\lambda) Q_0 + \lambda Q_1]^{-1} ((1-\lambda) Q_0\chi_0 + \lambda Q_1\chi_1). \label{eqn:QPexact}
\end{align}
The two limiting points of the Pareto optimal set are obtained as $x(0)=\chi_0$ (the unique minimum of $J_0$) and $x(1)=\chi_1$ (the unique minimum of $J_1$).

In order to trace the Pareto front using numerical integration as described in Sections~\ref{sec:Paretotracing} and \ref{sec:numericalintegration} above, the first order optimality conditions are differentiated w.r.t.\ $\lambda$, yielding the implicit ODE \eqref{eqn:implicit_bicriteria_ode} as 
\begin{equation*}
\frac{d}{d\lambda} \nabla_x J_\lambda(x(\lambda)) =  0 \quad\Leftrightarrow\quad 
((1-\lambda) Q_0 +\lambda Q_1)\dot{x}(\lambda) - Q_0(x(\lambda)-\chi_0) + Q_1(x(\lambda)-\chi_1) =  0
\end{equation*}
for $\lambda\in[0,1]$. Since $\nabla^2_x J_\lambda(x(\lambda))$ is positive definite, this can be rearranged to a standard ODE \eqref{eqn:bicriteria_ode} as follows:
\begin{equation}\label{eqn:QPDGL}
\dot{x}(\lambda) = ((1-\lambda)Q_0+\lambda Q_1)^{-1}
(Q_0(x(\lambda)-\chi_0) - Q_1(x(\lambda)-\chi_1)) = f(\lambda,x(\lambda)),\end{equation}
with possible initial values $x_0=x(\lambda_0)=\chi_0$ (for $\lambda_0=0$) or $x_0=x(\lambda_0)=\chi_1$ (for  $\lambda_0=1$).  

\begin{remark}
Since $\nabla^2_xJ_\lambda(x)=(1-\lambda)Q_0+\lambda Q_1$ is independent of $x$, and since it is positive definite for all $\lambda\in[0,1]$, its smallest eigenvalue $\Lambda(\lambda,x)$ is bounded below by a constant $\varepsilon>0$ on $[0,1]$, i.e., for $I=[0,1]$ we have  $\Lambda(I)=\inf_{\lambda\in I}\Lambda(\lambda,x(\lambda))\geq\varepsilon>0$. This implies that we can choose uniform constants $L_\sharp$, $\sharp= H,f,\lambda$, in Lemma~\ref{lem:f_is_Lipshitz} and Proposition~\ref{prop:eps-crit_front} where $L_H$ can be set to zero. Moreover, the above analysis shows that $f(\lambda,x(\lambda))$ is of class $C^\infty$ and hence high order iteration schemes are possible in this case, see Theorem~\ref{theo:order} and   Lemma~\ref{lem:regularity} above.
\end{remark}

Following \cite{tour:onbi:2019}, we generate random matrices $Q_{0}$ and $Q_{1}$, by $Q_j=M_j^\top M_j$, where $M_j$ is a sample from a $n\times n$-random matrix with independent standard normal distributed entries, $j=0,1$. Likewise, $\chi_j$ are $n$-dimensional random vectors with independent standard normal entries. We provide numerical tests for dimension $n=100$.  In Figure~\ref{fig:paretoCQMP} we compare the analytic solution $x(\lambda)$ with numerical solutions obtained from integrating \eqref{eqn:QPDGL} numerically with initial value $ x_0=x(0.5)$ given by the exact solution obtained from \eqref{eqn:QPexact}. We also apply a simple, gradient based descent algorithm using the Armijo rule (with parameter $\rho=0.5$) starting at $x_{0,0}=0$ in order to obtain approximate starting points $x_{0,k}$ to integrate for approximate Pareto solutions, as described in Proposition~\ref{prop:eps-crit_front}(ii).  
\begin{figure}[t]
    \centering
    \includegraphics[scale=.5]{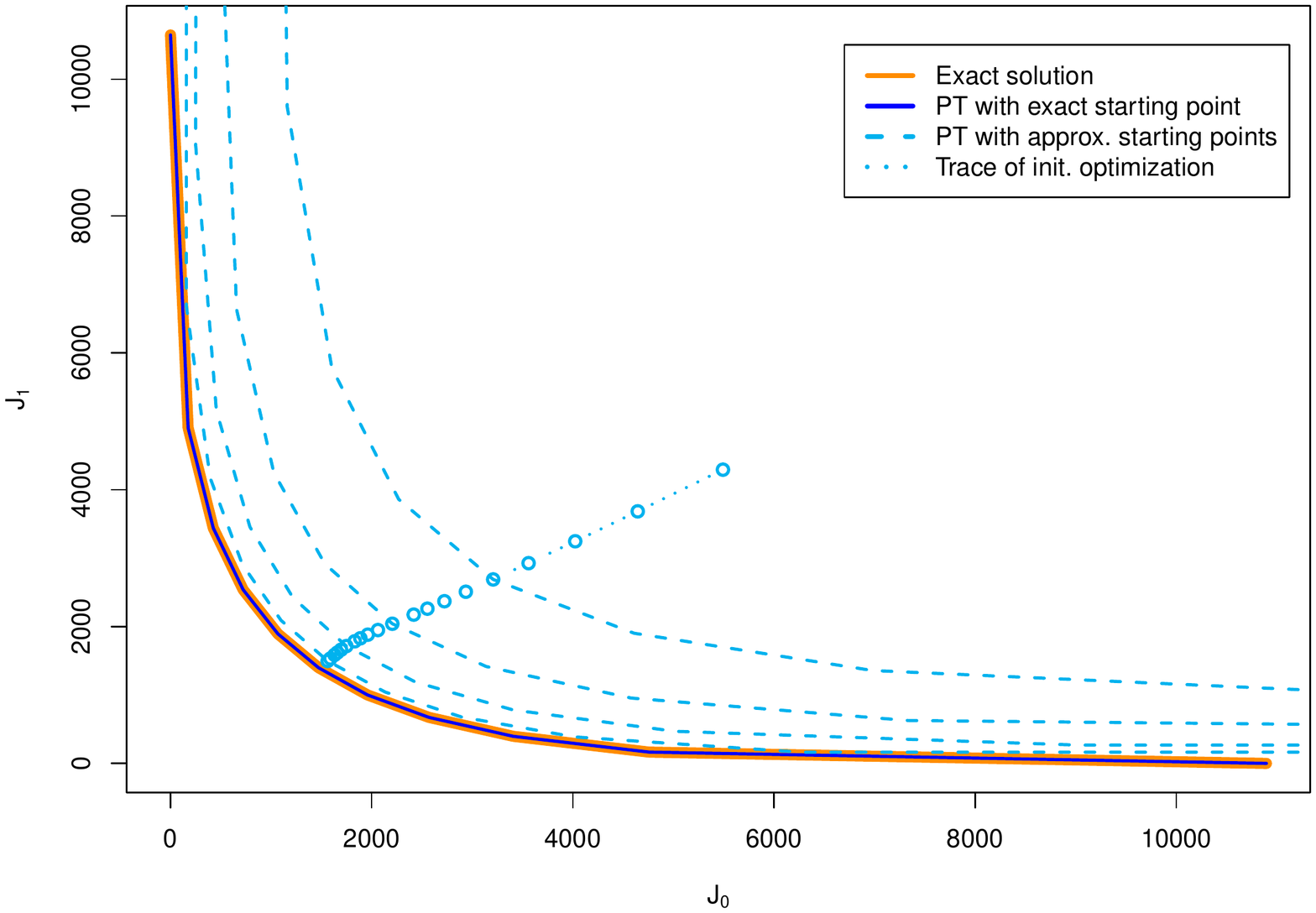}
    \caption{Comparison of the analytic solution \eqref{eqn:QPexact} for the Pareto front (orange, thick solid) with the objective values of the numerically integrated solution of the ODE \eqref{eqn:QPDGL} started at the exact solution (solid blue, $\lambda_0=0.5$) and integrated solutions (dashed light blue, $\lambda_0=0.5$) started at the 5th, 10th, 15th and 20th iteration of a gradient descent algorithm starting at $x_{0,0}=0$ (dotted light blue). The dimension of the problem is $n=100$. Numerical integration of the ODE uses 4th order Runge-Kutta method with $20$ iterations, $10$ in each direction (step length $h=0.05$).}
    \label{fig:paretoCQMP}
\end{figure}

\subsection{Pareto tracing for bi-criteria shape optimization}\label{subsec:shapes}
 In the following, the Pareto tracing approach is applied to  the bi-criteria shape optimization of a ceramic component under tensile load presented in  \cite{Doganay2019}. 
 As optimization criteria we consider the volume of the component on one hand, and its reliability on the other hand. Here, the reliability of the component is assessed via its probability of failure as introduced in \cite{bolten} and implemented for 2D shapes in \cite{PaperHahn}. 
We refer to \cite{Doganay2019} for a detailed derivation of the model and provide a brief summary below. 

Let $\Omega\subset \R^2$ be a compact body that is filled with ceramic material and has a piecewise Lipschitz boundary. We additionally assume that the boundary $\partial\Omega$ of $\Omega$ 
consists of three parts
\begin{align*}
	\partial\Omega = \text{cl}({\partial\Omega}_D) \cup \text{cl}({\partial\Omega}_{N_{\text{fixed}}}) \cup \text{cl}({\partial\Omega}_{N_{\text{free}}}),
\end{align*}
where on $\partial\Omega_D$ the Dirichlet boundary condition holds, the surface forces may act on $\partial\Omega_{N_{\text{fixed}}}$, and $\partial\Omega_{N_{\text{free}}}$ is the part that can be altered in an optimization approach. We further assume that a bounded open set $\widehat{\Omega} \subset \R^2$ that satisfies the \emph{cone property}, see, e.g., \cite{bolten}, contains all feasible shapes, see Figure~\ref{fig:adm_shape} for an example.

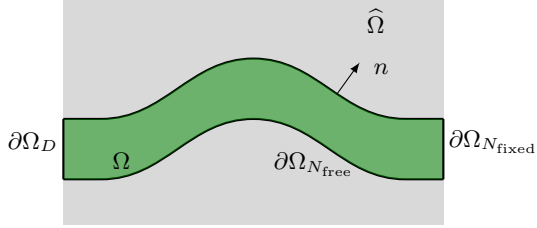
\begin{figure}[htb]
  \centering
  \begin{tikzpicture}[font={\footnotesize}]
    \coordinate (bb_ll) at (0,0.4);
    \coordinate (bb_lr) at (5,0.4);
    \coordinate (bb_tl) at (0,3.4);
    \coordinate (bb_tr) at (5,3.4);
    \coordinate (bar_ll) at (0,1.);
    \coordinate (bar_lr) at (5,1.);
    \coordinate (bar_tl) at (0,1.8);
    \coordinate (bar_tr) at (5,1.8);
    \fill [fill=black!15!white] (bb_ll) rectangle (bb_tr);
    \draw [thick,name path=A] (bar_ll) to[out=0,in=180] (0.5,1) to[out=0,in=180] (2.5,1.8) to[out=0,in=180] (4.5,1) to[out=0,in=180] (bar_lr);
    \draw [thick,name path=B] (bar_tl) to[out=0,in=180] (0.5,1.8) to[out=0,in=180] (2.5,2.6) to[out=0,in=180] (4.5,1.8) to[out=0,in=180] (bar_tr);
    \tikzfillbetween[of=A and B] {black!45!green, opacity=0.5};
    \node[label=left:\(\widehat{\Omega}\)] at (4.5,3.1) {};
    \node[label=left:\(\Omega\)] at (1.15,1.25) {};
    \node[label=right:\(\partial\Omega_{N_{\text{fixed}}}\)] at (4.8,1.5) {};
    \node[label=right:\(\partial\Omega_{N_{\text{free}}}\)] at (2.5,1.2) {};
    \node[label=left:\(\partial\Omega_{D}\)] at (0.2,1.5) {};
    \draw [thick] (bar_ll) -- (bar_tl) ;
    \draw [thick] (bar_lr) -- (bar_tr) ;
    \draw [-latex] (3.6,2.13) -- (3.9,2.55) node [label={[xshift=8pt,yshift=-12pt]\({n}\)}] {} ;
  \end{tikzpicture}
  \caption{Illustration of a possible admissible shape \(\Omega\in \Oad\), see also Figure~1 in \cite{Doganay2019}.\label{fig:adm_shape}}
\end{figure}
An admissible shape is then defined as an element of the set
\[\Oad:=\{\Omega \subset \widehat{\Omega}:\; {\partial\Omega}_D \subset {\partial\Omega},\; {\partial\Omega}_{N_{\text{free}}} \subset \partial\Omega,\; \widehat{\Omega} \text{ and }\Omega\text{ satisfy the cone property}\}.\]

Since ceramics behave according to linear elasticity theory, see, e.g., \cite{braess,munz}, one can express the state equation which describes the behaviour of the ceramic component under external forces like tensile load by an elliptic partial differential equation as follows:
\begin{equation} \label{stateequation}
\begin{array}{rcll}    
  -\text{div}(\sigma(u(z))) & = & \bar{f}(z) &  \text{for} \; z\in\Omega \\    
  u(z) & = & 0 &  \text{for} \; z\in\partial\Omega_D    \\
 \sigma(u(z)){n}(z) & = & \bar{g}(z) &  \text{for} \; z\in\partial\Omega_{N_{\text{fixed}}} \\
  \sigma(u(z)){n}(z) & = & 0 & \text{for} \; z\in\partial\Omega_{N_{\text{free}}} 
\end{array}
\end{equation}
Here, $\bar{f}\in L^2(\Omega , \R^2)$ represents the volume forces and $\bar{g} \in L^2(\partial\Omega_{N_{\text{fixed}}} , \R^2)$ the forces acting on the surface $\partial \Omega_{N_{\text{fixed}}}$. The resulting displacement of the component is given by $u\in H^1(\Omega, \R^2)$, and $D u$ denotes the Jacobian of $u$. Then the linear strain tensor $\varepsilon\in L^2(\Omega, \R^{2\times 2})$ has the form $\varepsilon(u(z)):=\frac{1}{2}(D u(z) + (D u(z))^{\top})$. With the Lam\'e constants $\hat{\lambda}=\frac{\nu E}{(1+\nu)(1-2\nu)}$ and $\hat{\mu}=\frac{E}{2(1+\nu)}$ obtained from Young's modulus $E$ and Poisson's ratio $\nu$ the
stress tensor $\sigma\in L^2(\Omega, \R^{2\times 2})$ is given by $\sigma(u(z))=\hat{\lambda}\, \text{tr}(\varepsilon(u(z)))I+ 2\hat{\mu}\varepsilon(u(z))$. The outward pointing normal at $z\in\partial\Omega$ is denoted by ${n}(z)$ and is defined nearly everywhere on $\partial \Omega$.

Now the considered \emph{bi-criteria shape optimization problem} can be formulated as
\begin{equation}
\begin{split}
\min_{\Omega \in \Oad} & ~J(\Omega):=(J_0(\Omega),J_1(\Omega))\\
\text{s.t. } & u \in H^1(\Omega, \R^2) \text{ solves the state equation } (\ref{stateequation}),
\end{split}\label{ceramicMOP}
\end{equation}
where $J_0(\Omega):=\int_{\Omega}\text{d}z$ denotes the volume of the shape $\Omega\in\Oad$ and  $J_1(\Omega)$ is an intensity measure modelling its probability of failure. 
Following \cite{bolten, Doganay2019} we model the probability of failure as an intensity measure of a Poisson point process which counts the ``critical'' cracks in a ceramic component, where ``critical'' means that these cracks may initiate ruptures under tensile load. This leads to the following Weibull-type functional which is used as the second objective in our study: 
\begin{align*}
J_1(\Omega):=\frac{1}{2\pi}\int\limits_{\Omega}\int\limits_{S^1}\left(\frac{\Bigl(n^\top \sigma(Du(z))n \Bigr)^+}{\sigma_0}\right)^m \text{d}n\,\text{d}z .
\end{align*}
Here, $(\cdot)^+:=\max(\cdot, 0)$, $S^1$ is the unit sphere in $\R^2$, $\sigma_0$ is a positive constant and the parameter $m$ is called \emph{Weibull module} and  typically assumes values between $5$ and $25$. We refer to \cite{bolten} for further details.

The implementation of \cite{PaperHahn} is used to evaluate the objectives and gradients. It is based on standard Lagrangian finite elements to discretize a two-dimensional shape $\Omega \in \Oad $ by an $n_x\times n_y$ finite element mesh \(Z:=(Z^{\Omega}_{ij})_{n_x\times n_y}\). 
Numerical quadrature is used to calculate all occurring integrals and an adjoint approach is utilized to speed up the computation of the gradients. 
We adopt the geometry definition of \cite{Doganay2019} 
that takes advantage of the geometry of the considered shapes to reduce the number of variables. In a first step, all $x$-components of the grid points are fixed and \emph{mean line} and \emph{thickness} values  $\varrho^{\text{ml}}\in\R^{ n_x}$ and $\varrho^{\text{th}}\in\R^{ n_x}_+$
are used to represent the discretized shape $Z$. Since we observed in \cite{Doganay2019} that, when starting from a reasonable initial solution, nonnegativity constraints on the thickness values are automatically satisfied during all iterations of the optimization process we omit these constraints in the following. The numerical tests presented below confirm this observation.  
In a second step, these meanline and thickness values are fitted with B-splines with a prespecified number of  $n_B$  
basis functions $\vartheta_j,\ j=1,\hdots, n_B$, yielding smoothed meanline and thickness values by computing
\begin{equation*}
\hat{\varrho}^{\text{ml}}(z):=\sum_{j=1}^{n_B} x^{\text{ml}}_j \, \vartheta_j(z) \quad \text{and}\quad \hat{\varrho}^{\text{th}}(z):=\sum_{j=1}^{n_B} x^{\text{th}}_j \, \vartheta_j(z), \qquad z\in \R,
\end{equation*}
see, e.g., \cite{nurbs}. 
The B-spline coefficients $x=(x^{\text{ml}}, x^{\text{th}}) \in \R^{ n_B}\times \R^{ n_B}_+$ are then used as optimization variables in \eqref{ceramicMOP}, replacing $J_i(\Omega)$ by $J_i(Z)\approx J_i(x)$, $i=0,1$. 

In order to trace the Pareto front of the bi-criteria shape optimization problem \eqref{ceramicMOP} with the methodology described in the previous sections, we first have
to compute the right hand side $f(\lambda,x(\lambda))$ of \eqref{eqn:bicriteria_ode} for problem \eqref{ceramicMOP}. Towards this end, the $n_B \times n_B$ Hessian matrix $\nabla^2_xJ_i(x)$, 
$i=0,1$, is approximated by the finite difference method with a precision of $\varepsilon_H=10^{-6}$. Note that this can be done in parallel. For the stability of our approach under an approximate evaluation of $f(\lambda,x(\lambda))$, we refer to Proposition \ref{prop:eps-crit_front} (iii). 
For the numerical solution of the ODE \eqref{eqn:bicriteria_ode} we apply an order $2$ Runge-Kutta method, thus requiring that the discretized objective functions $J_i$, $i=0,1$, are at least $4$-times continuously differentiable (c.f.\ Theorem~\ref{theo:order}). 
This is clearly satisfied for the discretized volume $J_0(x)$ which is, as a polynomial, infinitely differentiable. For the discretized intensity measure $J_1(x)$, we can build on the analysis for $J_1(Z)$ performed in \cite{PaperHahn,gottschalk2019shape}. Here, the discretized state equation \eqref{stateequation} is of the form $B(Z)\,U(Z)=F(Z)$, where $U(Z)$ is the discretized displacement, $B(Z)$ the positive definite stiffness matrix and $F(Z)$ the discretized forces. From the assembly of $B(Z)$ and $F(Z)$ in \cite{PaperHahn,gottschalk2019shape} it can be seen that $B(Z),F(Z) \in C^{\infty}$. Using the identity $U(Z)=B(Z)^{-1}F(Z)$, where the right hand side is infinitely differentiable, it can be shown  iteratively that also $U(Z) \in C^{\infty}$. Moreover, in \cite[Lemma~6.5.5]{Bittner} it is shown that $\zeta(\sigma)=((n^\top \sigma\, n)^+)^m$ is $m$ times continuously differentiable w.r.t.\ $\sigma$. We can conclude that this is also the case for $J_1(Z)$. The order $2$ Runge-Kutta method is hence applicable for Weibull modules $5 \leq m\leq 25$.

Initial values for Pareto front tracing can be obtained by any of the methods suggested in \cite{Doganay2019}. In the following case study, one  weighted sum scalarization of the bi-criteria shape optimization problem \eqref{ceramicMOP} is solved using a gradient descent algorithm with Armijo step lengths.

\subsubsection*{Test Cases}
We consider the same 2D test cases that were investigated in \cite{Doganay2019} to obtain comparable results. The two 2D shapes are made from ceramic beryllium oxide (BeO) and are under tensile load. The material parameters of BeO are set according to \cite{munz,crcmaterials}, i.e., Poisson's ratio is set to $\nu=0.25$, Young's modulus to $\texttt{E}=320\,\text{GPa}$ and the ultimate tensile strength to $140\,\text{MPa}$. We choose $m=5$ for the Weibull module. Moreover, both shapes have a fixed height of $0.2\,\text{m}$ on the left and right boundaries and a fixed length of  $1.0\,\text{m}$. The left boundary corresponds to the Dirichlet boundary $\partial\Omega_{D}$, i.e., the boundary is fixed and no forces act on it, while the right boundary is also fixed but corresponds to the Neumann boundary $\partial\Omega_{N_{\text{fixed}}}$, i.e., the surface forces $\bar{g}$ act on that boundary. The remaining upper and lower boundaries correspond to the part that is force free, i.e., $\partial\Omega_{N_{\text{free}}}$, which can be modified during the optimization process.  Following \cite{Doganay2019}, we neglect the gravity forces (i.e., $\bar{f} = 0$) and set the tensile load to $\bar{g} = 10^{7}\,\text{Pa}$. 

Note that individual optima for $J_0$ and $J_1$ do not exist under these assumptions. Indeed, the infimum of the volume $J_0$ is zero, and hence optimal shapes do not exist when minimizing $J_0$ without any additional constraints. Conversely, when considering solely the probability of failure $J_1$, then the reliability can always be improved when increasing the volume (as long as we neglect gravity forces). As a consequence, the weighted sum scalarization $J_{\lambda}$ can only have solutions for weights $\lambda\in(0,1)$, where we can expect problems the closer $\lambda$ gets to either boundary of this interval. This is confirmed by the numerical tests presented below.

The shapes are discretized using a triangular $41\times 7$ mesh, i.e.,  $n_x=41$ and $n_y=7$. Meanline and thickness values are fitted with B-splines with $n_B=5$ basis functions, yielding ten B-spline coefficients in total. Since the coefficients that correspond to the fixed boundaries are fixed, this results in six optimization variables, c.f.\  \cite{Doganay2019}. All numerical experiments are realized in R (version 3.5) using the implementation of \cite{PaperHahn} to compute the objective values and the (adjoint) gradients on the mesh. 
We use the implementation of the Runge-Kutta method provided by the R package ``deSolve'' to solve the resulting ODE.

\subsubsection*{Test Case 1: A Straight Joint} 
\label{subsubsec:StraightJoint}

For the first test case we fix the left and right boundaries at the same height and apply the surface forces $\bar{g}$ on the right boundary. Under these circumstances, straight rods with varying thickness that connect the boundaries can be expected as solutions of the bi-criteria shape optimization problem \eqref{ceramicMOP}. The numerical studies in \cite{Doganay2019} support this intuition, see Figure~\ref{fig:TC1_shapes} for some exemplary results.

\begin{figure}[ht]
	\begin{center}
		\subfloat[$\lambda=0.2$ ]{\includegraphics[width=0.25\textwidth]{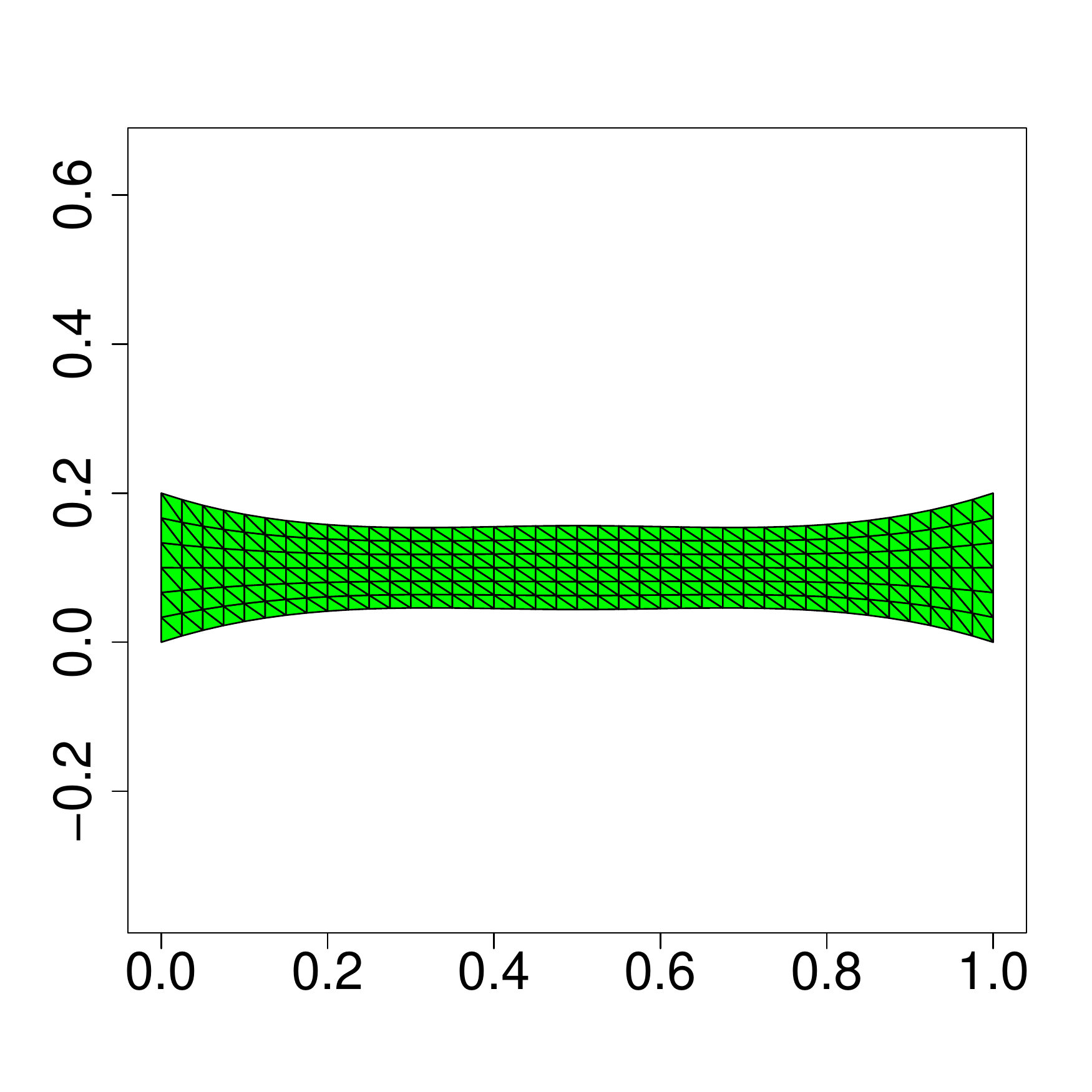}}
		\subfloat[$\lambda=0.5$
		]{\includegraphics[width=0.25\textwidth]{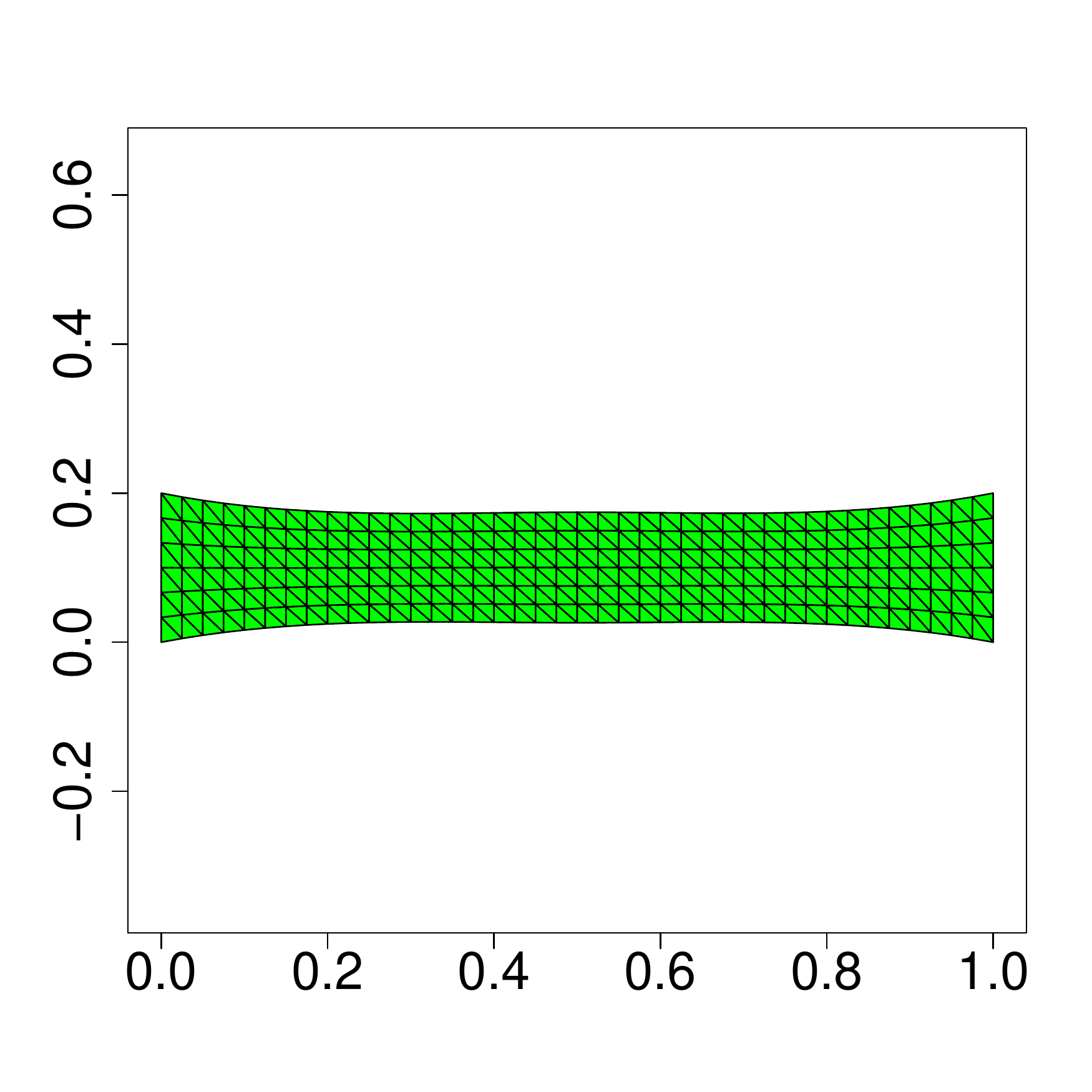}}
		\subfloat[$\lambda_0=0.813;\  x_{0}$ \label{fig:TC1_ODE_straight} ]{\includegraphics[width=0.25\textwidth]{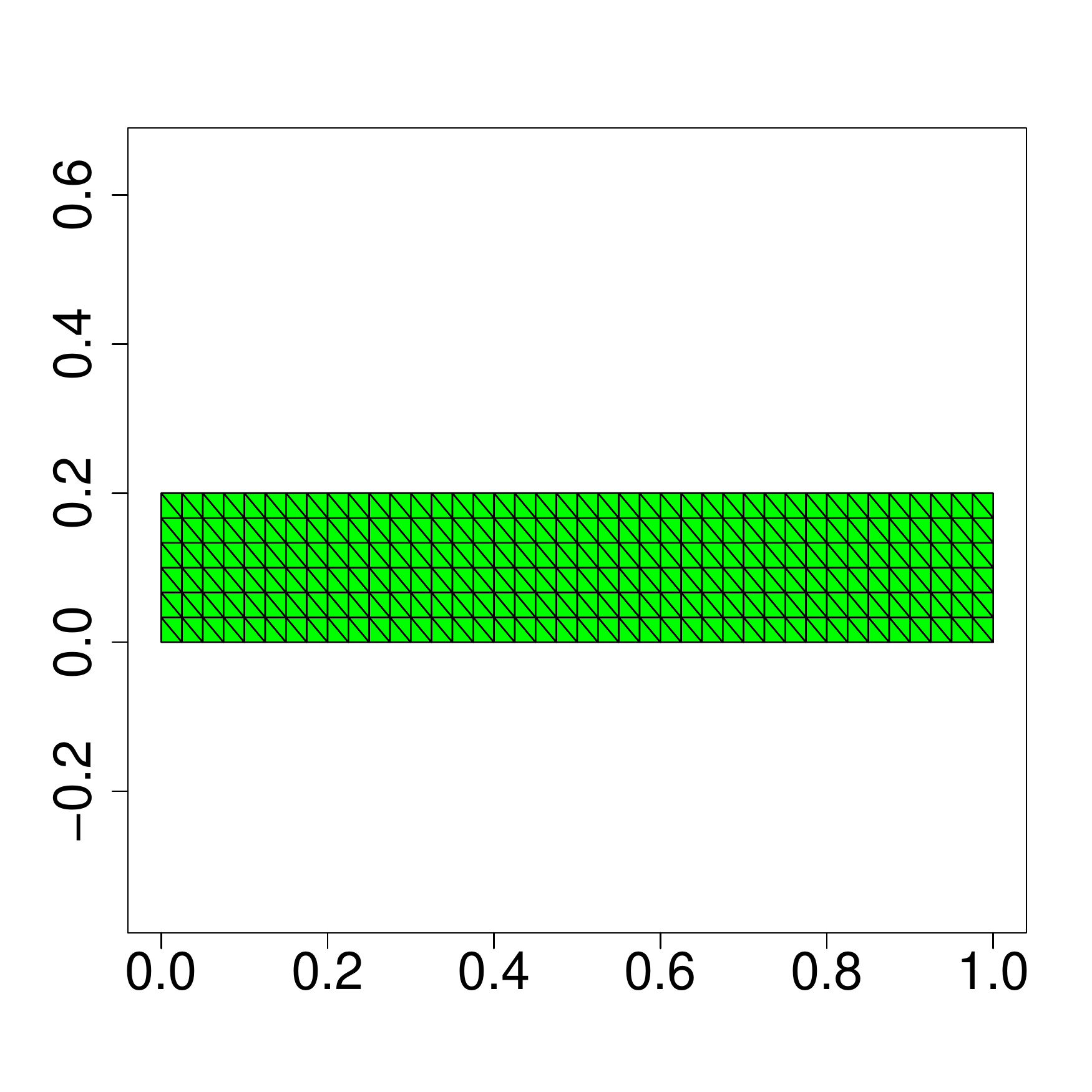}}
		\subfloat[$\lambda=0.9$
		]{\includegraphics[width=0.25\textwidth]{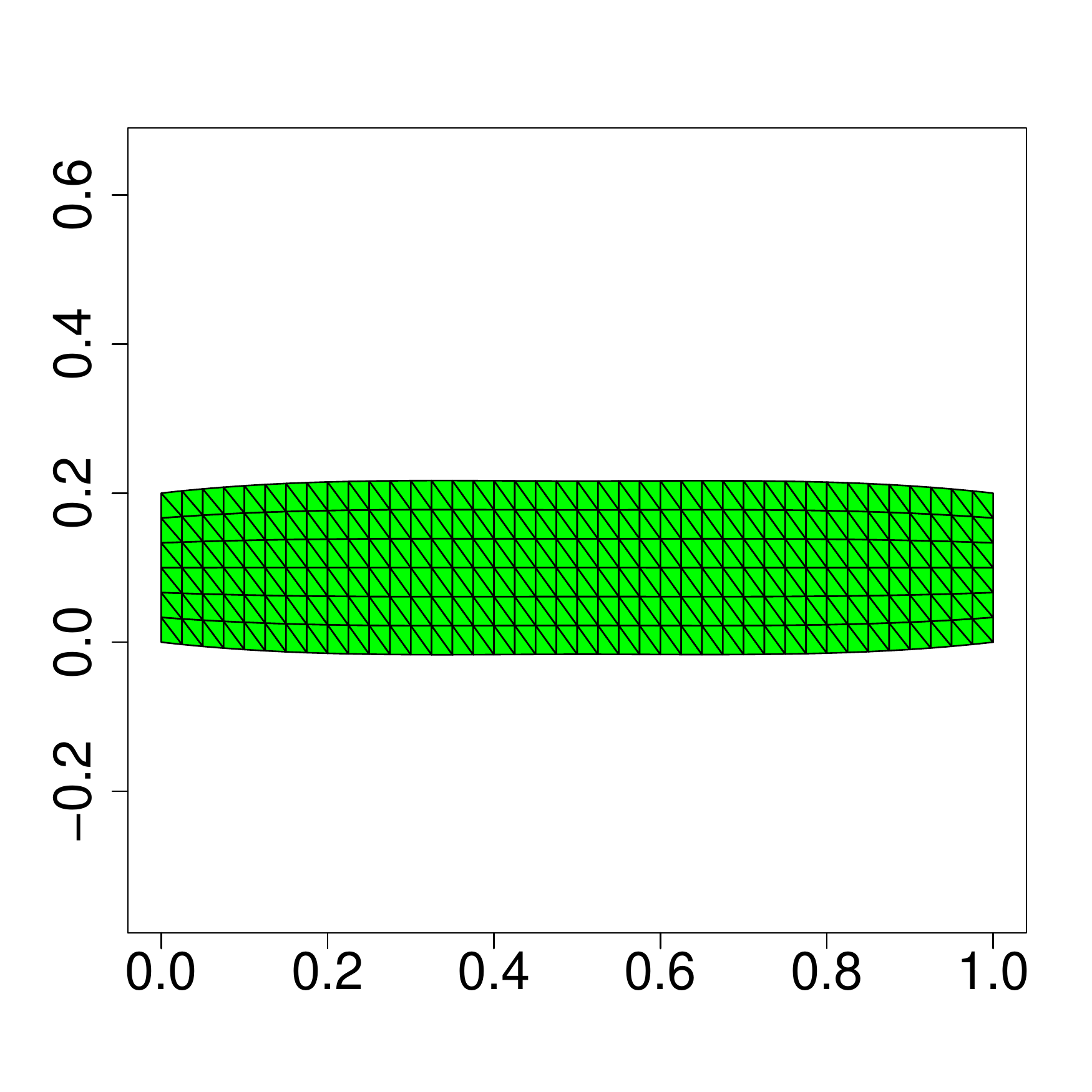}}

 	\end{center}		
	\caption{Exemplary solutions of the weighted sum method of \cite{Doganay2019} and the initial shape $x_{0}$. \label{fig:TC1_shapes}}
\end{figure}

This is the motivation for using a discretized straight rod with constant thickness of $0.2\,\text{m}$ as the initial shape $x_0$ for Pareto front tracing, see Figure~\ref{fig:TC1_ODE_straight}, even though this particular shape was not the outcome of any weighted sum scalarization considered in \cite{Doganay2019}. To determine a corresponding weight $\lambda_0$ such  that $x_0$ is $J_{\lambda_0}$-critical the equation $\| \nabla J_{\lambda}(x_0) \|=0$ is solved for $\lambda_0\in(0,1)$. The resulting weight has the value $\lambda_0\approx 0.813$ and we therefore have $x_0\approx x(0.813)$. 
Numerical integration in $[\lambda_l,\lambda_u]=[\lambda_0-0.66, \lambda_0+0.1]$ with a step size of $h=0.01$ resulted in solutions of varying thickness that are also straight rods and hence coincide with the results of \cite{Doganay2019}, see Figure~\ref{fig:TC1_ODE_shapes} for some exemplary shapes corresponding to those from Figure~\ref{fig:TC1_shapes}.

\begin{figure}[ht]

    \centering
	\begin{tikzpicture}[x=\textwidth]
		\node [label=below:{$x(0.203)$}] (1) at (0,0) {\includegraphics[width=3cm,keepaspectratio]{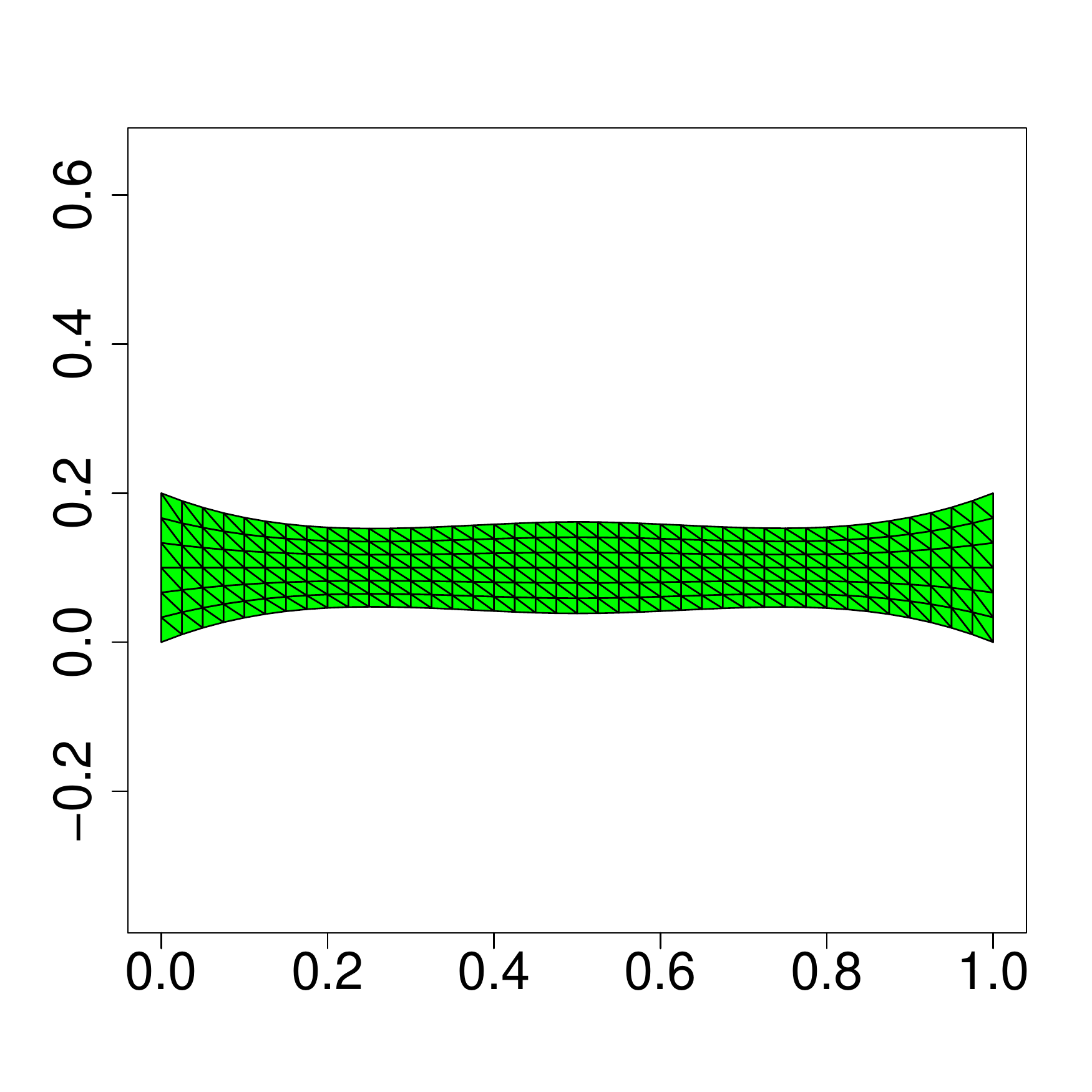}};
		\node [label=below:{$x(0.503)$}] (2) at (0.25,0) {\includegraphics[width=3cm,keepaspectratio]{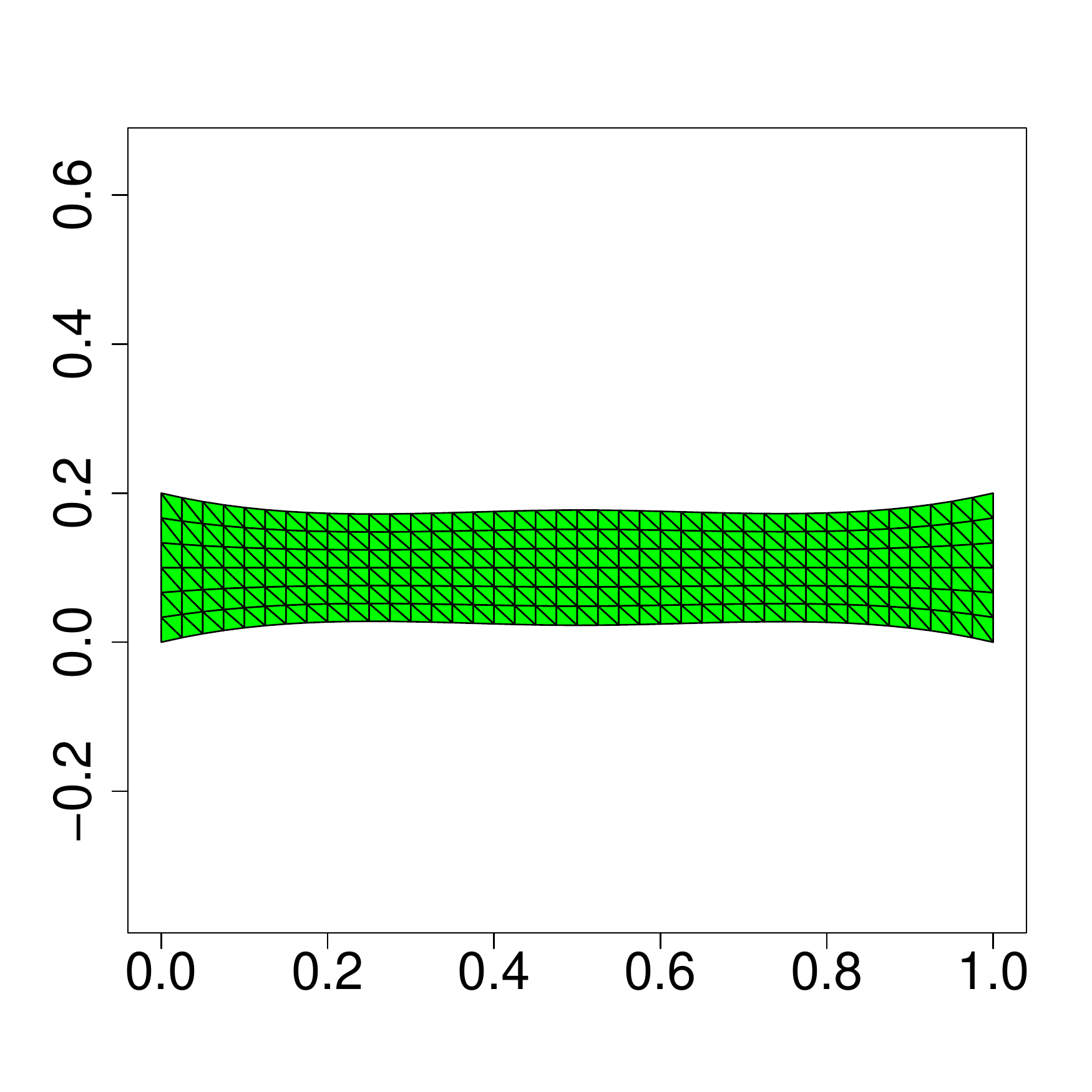}};
		\node [label=below:{$x_{0}$}] (3) at (0.5,0) {\includegraphics[width=3cm,keepaspectratio]{Figures/ShapeOpt/RM01/StraightRod.pdf}};
		\node [label=below:{$x(0.903)$}] (4) at (0.75,0) {\includegraphics[width=3cm,keepaspectratio]{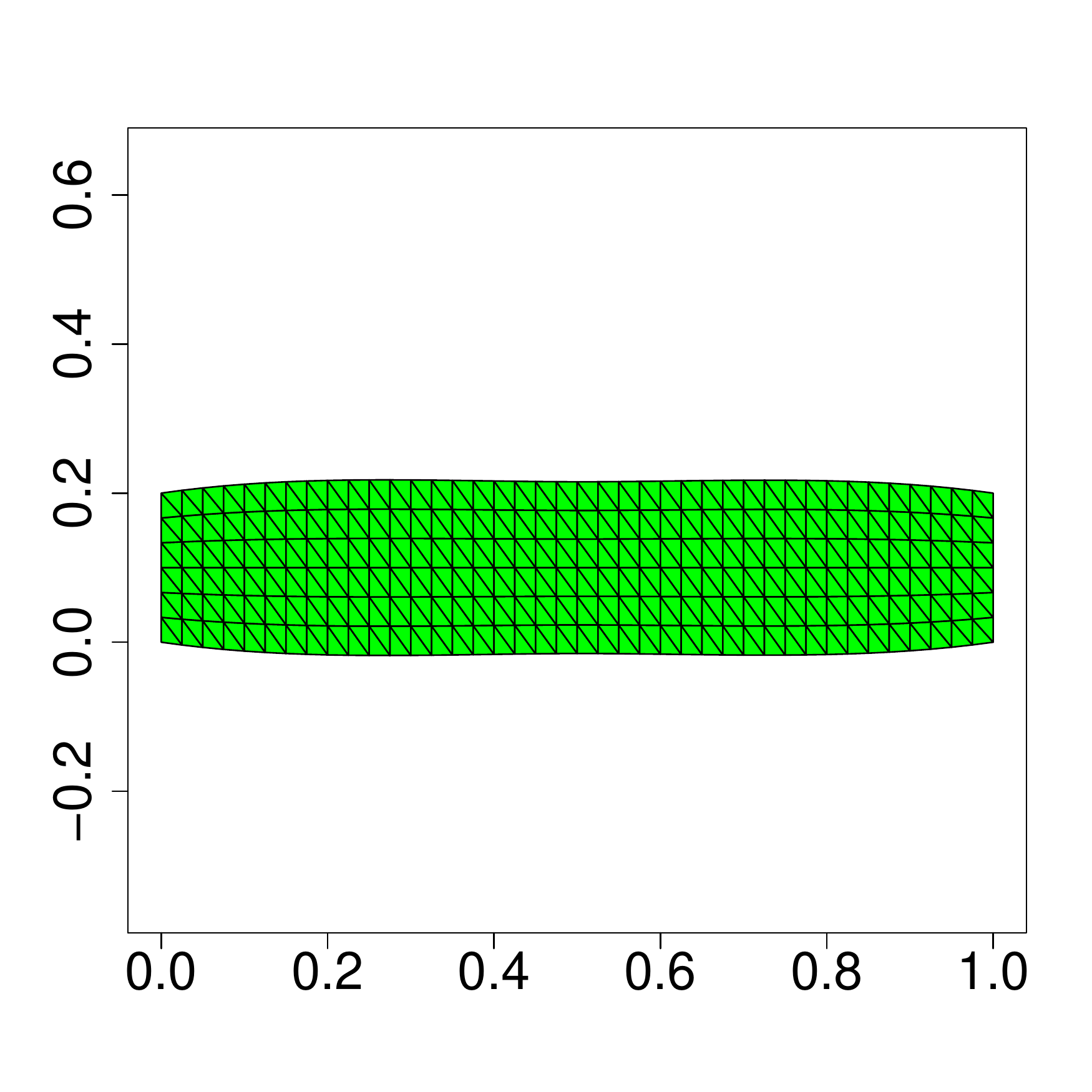}};
		
		\draw[->] (2) -- (1) node[midway, above = 0.1cm, sloped] {\tiny{Int. in }} node[midway, below = 0.1cm, sloped] {\tiny{neg. dir.}};
		\draw[->] (3) -- (2) node[midway, above = 0.1cm, sloped] {\tiny{Int. in}} node[midway, below = 0.1cm, sloped] {\tiny{neg. dir.}};
		\draw[->] (3) -- (4) node[midway, above = 0.1cm, sloped] {\tiny{Int. in}} node[midway, below = 0.1cm, sloped] {\tiny{pos. dir.}};
	\end{tikzpicture}	
	
    \caption{\label{fig:TC1_ODE_shapes} Exemplary results of the numerical integration of the ODE \eqref{eqn:bicriteria_ode} in positive and negative direction, starting from $x_{0}$.}

\end{figure}

In Figure~\ref{fig:TC1_CompareFronts} the outcome vectors obtained from numerical   integration are compared in the objective space with the outcome vectors obtained in \cite{Doganay2019} from the repeated solution of  weighted sum scalarizations using a  gradient descent algorithm. The results nicely document that the Pareto tracing approach not only covers the weighted sum solutions, but also  approximates a larger part of the (local) Pareto front. 

\begin{figure}[ht]
    \centering
    \includegraphics[scale=.6]{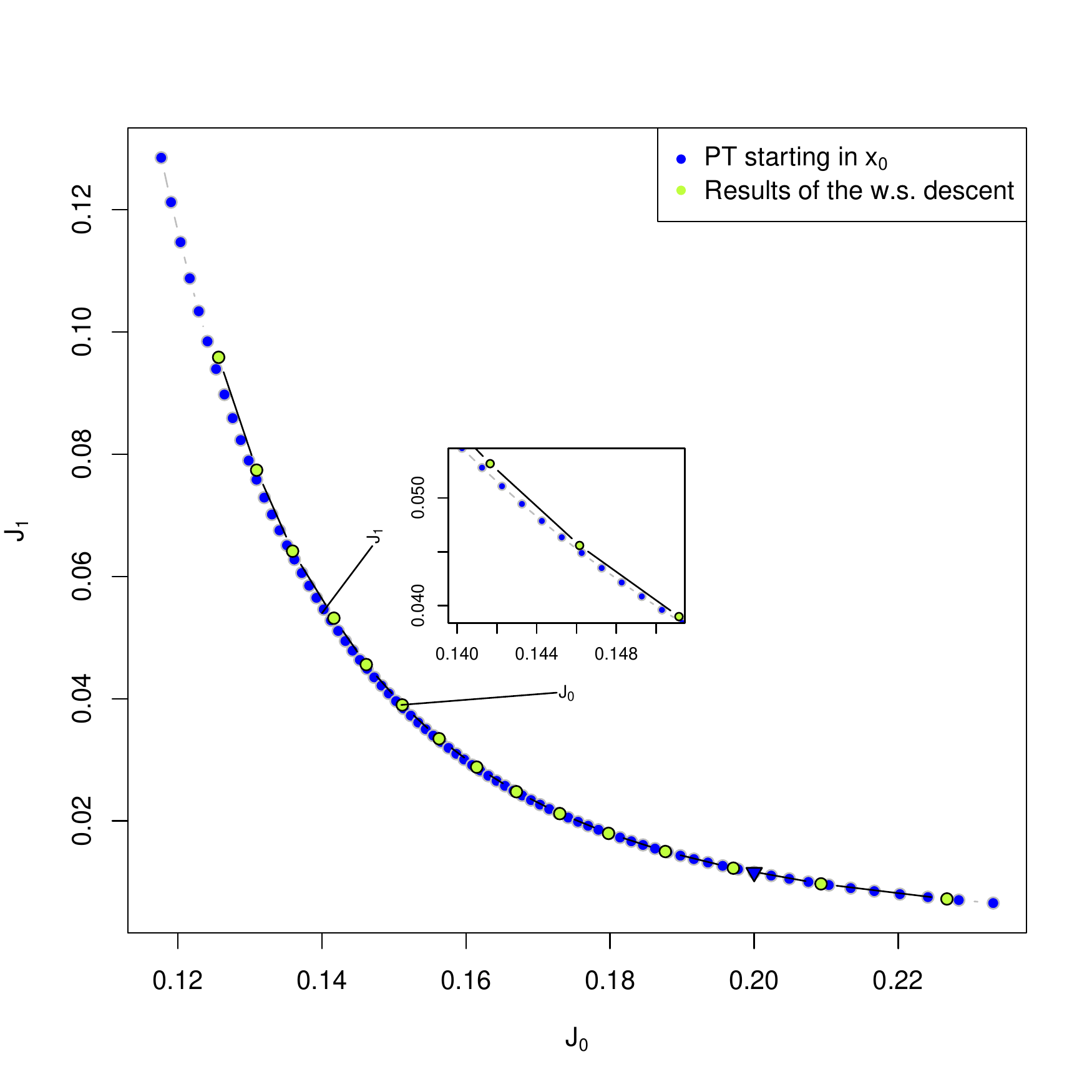}
    \caption{Comparison of the outcome vectors of Pareto tracing (blue) starting in $x_{0}$ and the outcome vectors obtained from the repeated application of gradient descent in \cite{Doganay2019} (green).}
    \label{fig:TC1_CompareFronts}
\end{figure}{}

Figure~\ref{fig:TC1_Optimality} shows the results from evaluating  
the first and the second order optimality conditions during the course of Pareto tracing, validating the statement of Proposition~\ref{prop:eps-crit_front}(ii). Indeed, the results nicely show that the computed shapes consistently achieve good  w.r.t.\ first and second order optimality tests.

\begin{figure}[ht]
	\begin{center}
		\subfloat[First order optimality: $\|\nabla_xJ_{\lambda}(x)\|$ ]{\includegraphics[width=0.5\textwidth]{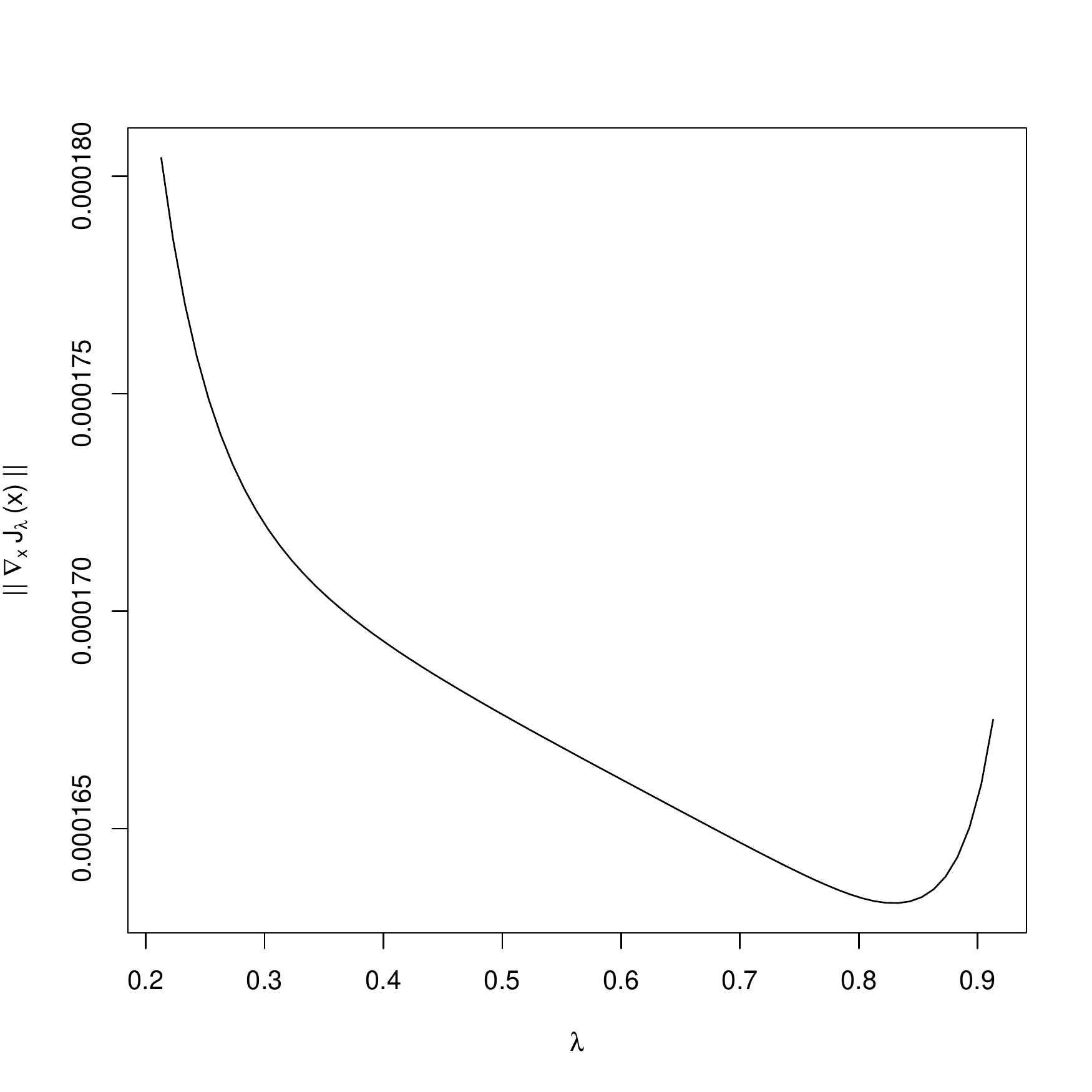}}
		\subfloat[Second order optimality: smallest eigenvalue  $\Lambda(\lambda,x(\lambda))$ of $\nabla_x^2J_{\lambda}(x)$
		]{\includegraphics[width=0.5\textwidth]{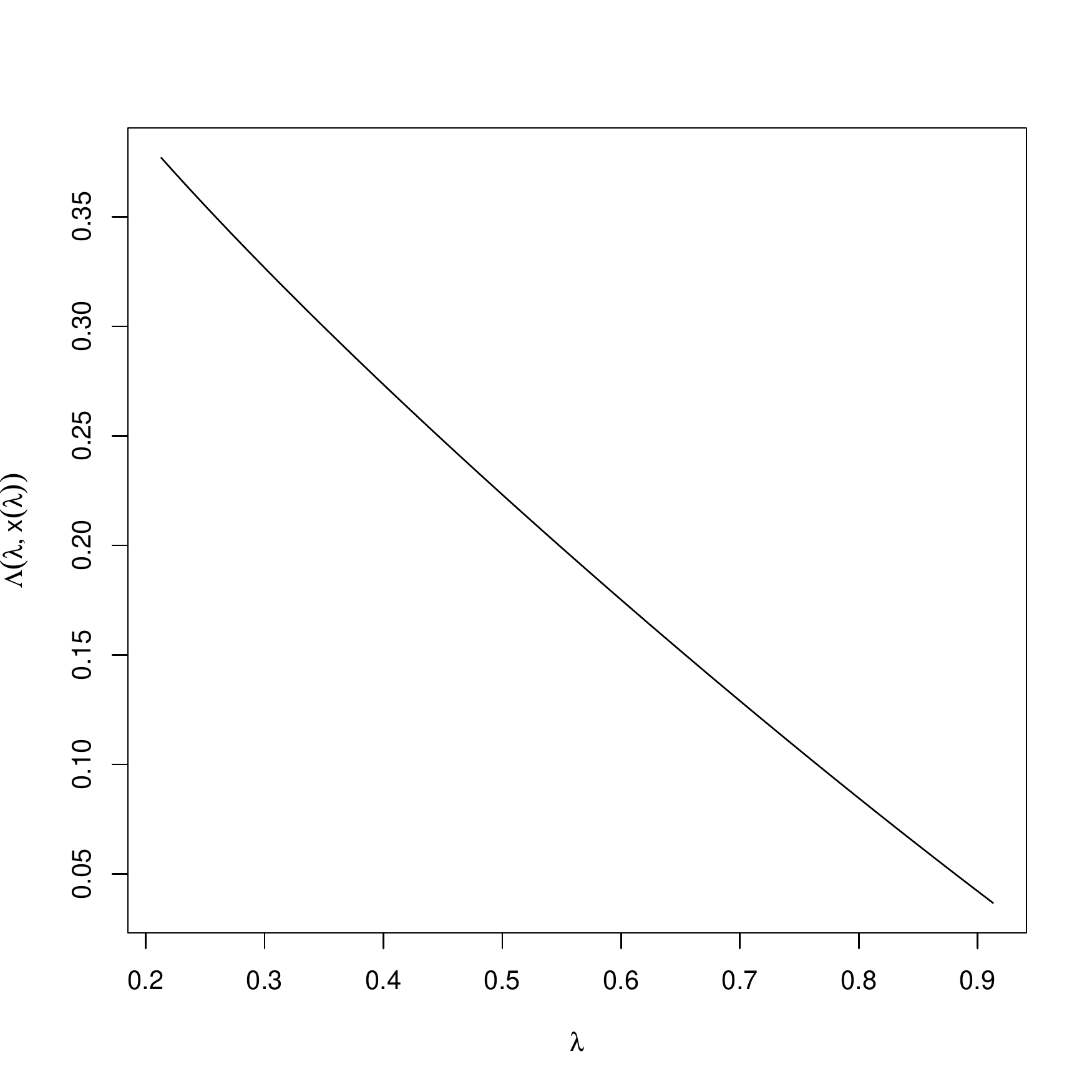}}
 	\end{center}	
	\caption{Straight Joint: Evaluating first and second order optimality during Pareto tracing \label{fig:TC1_Optimality}	}
\end{figure}

\subsubsection*{Test Case 2: An S-Shaped Joint} 
\label{subsubsec:S-Joint}
In the second test case the right boundary is placed about $0.27\,\text{m}$ lower than the left boundary, and hence an S-shaped joint is sought rather than a straight joint. In this case, the optimal shapes are not obvious. The numerical studies of \cite{Doganay2019} suggest that the (locally) Pareto optimal shapes resemble the profiles of whales with varying volume.  
Figure~\ref{fig:TC2_WS} shows exemplary solutions from \cite{Doganay2019} obtained from solving weighted sum scalarizations with weights $\lambda=0.25,0.4,0.6,0.8$. For $\lambda<0.25$ and $\lambda > 0.8$ the gradient descent method did not converge and hence we omit these solutions for the comparison.

\begin{figure}[ht]
	\begin{center}
		\subfloat[$\lambda=0.25;\ x_{0,k',0.25}$\label{fig:TC2_WS_w025} ]{\includegraphics[width=0.25\textwidth]{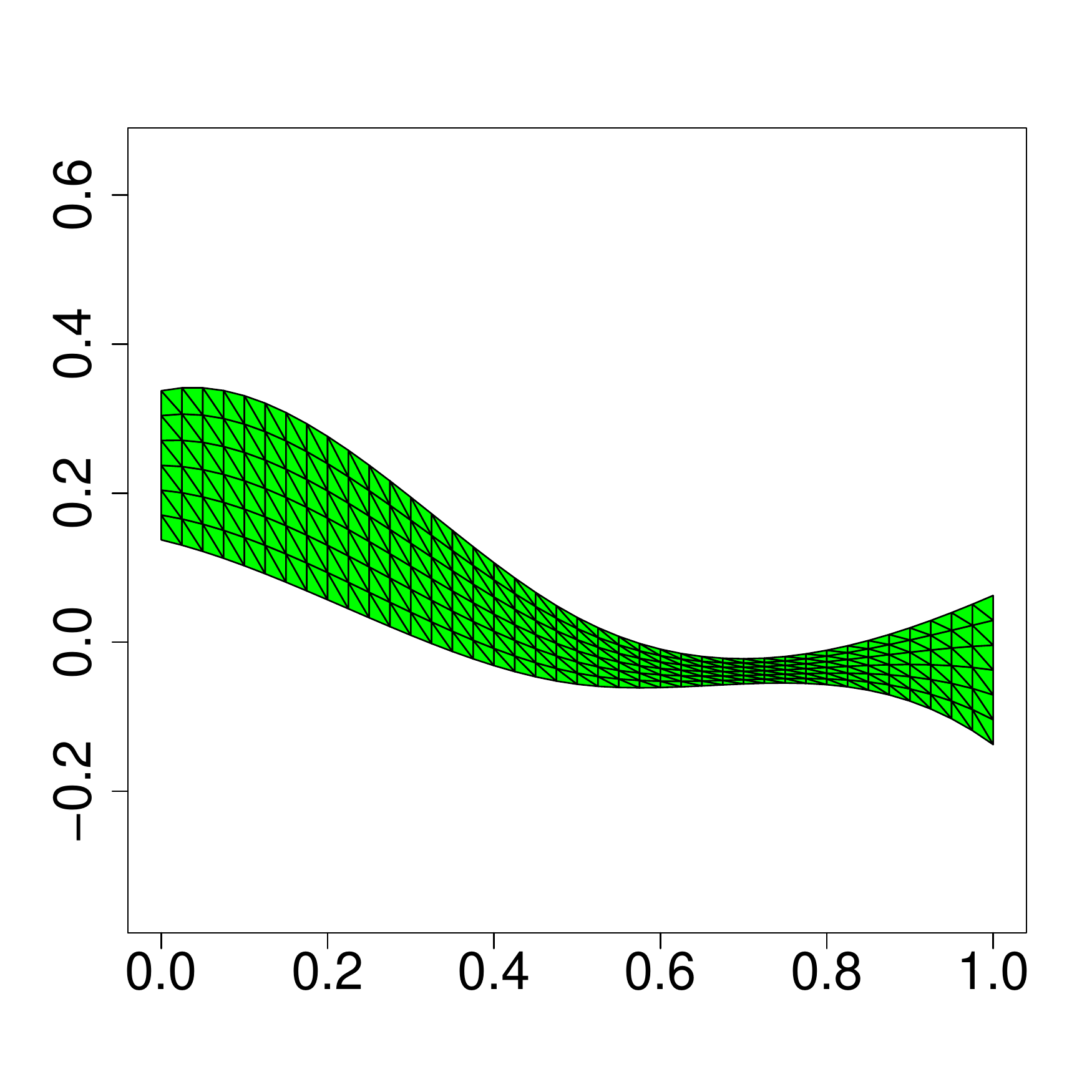}}
		\subfloat[$\lambda=0.4$ ]{\includegraphics[width=0.25\textwidth]{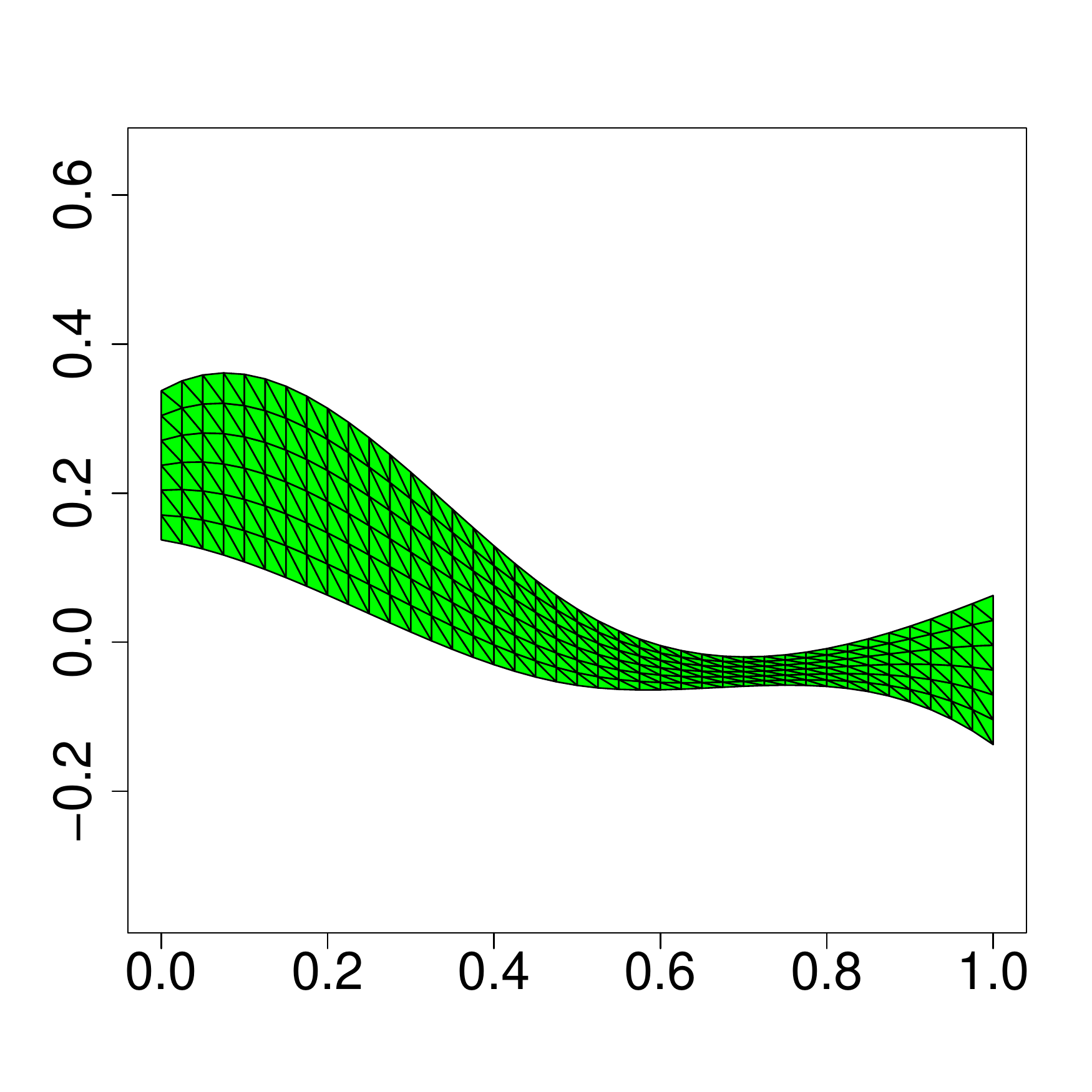}}
		\subfloat[$\lambda=0.6$  ]{\includegraphics[width=0.25\textwidth]{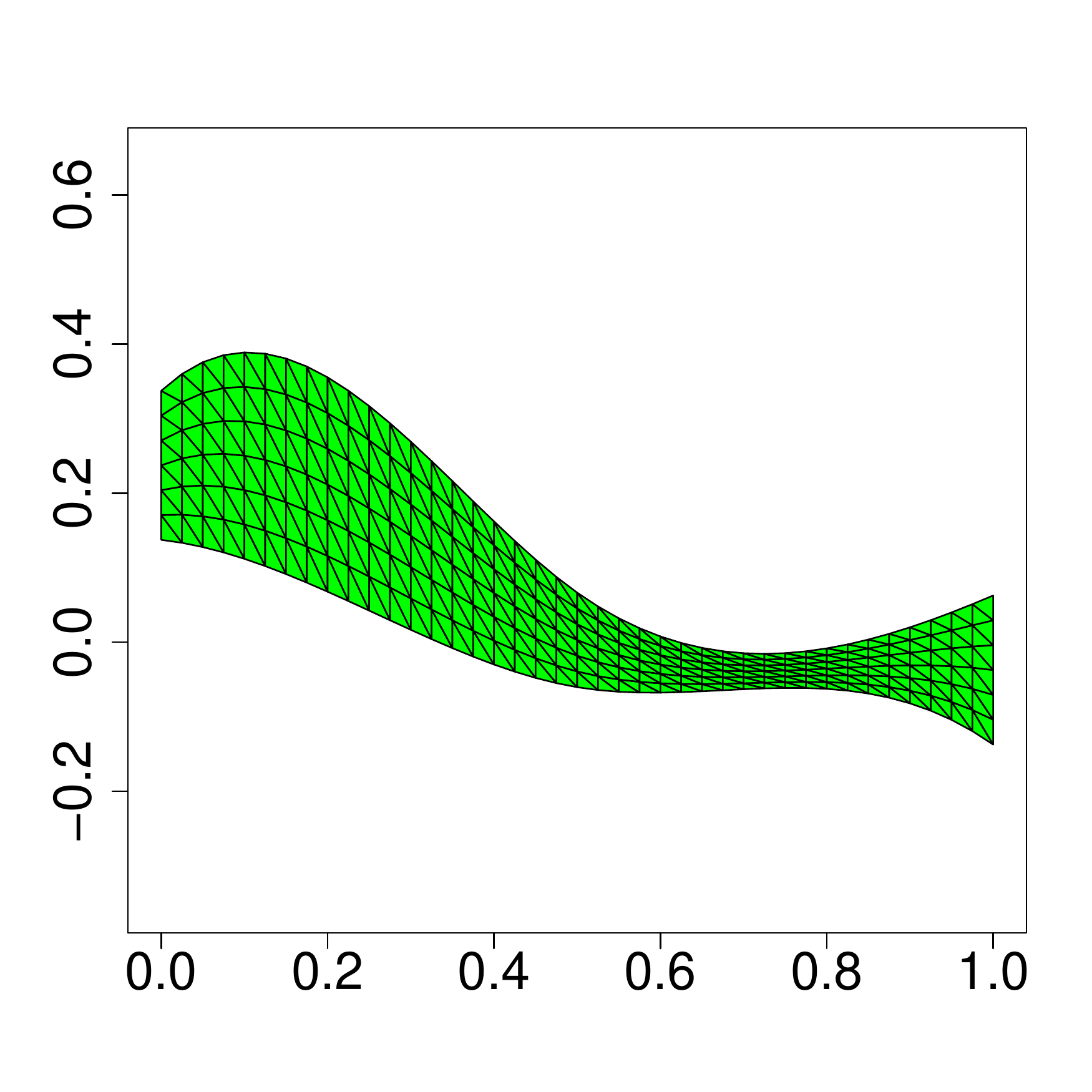}}
		\subfloat[$\lambda=0.8;\ x_{0,k'',0.8}$\label{fig:TC2_WS_w080}
		]{\includegraphics[width=0.25\textwidth]{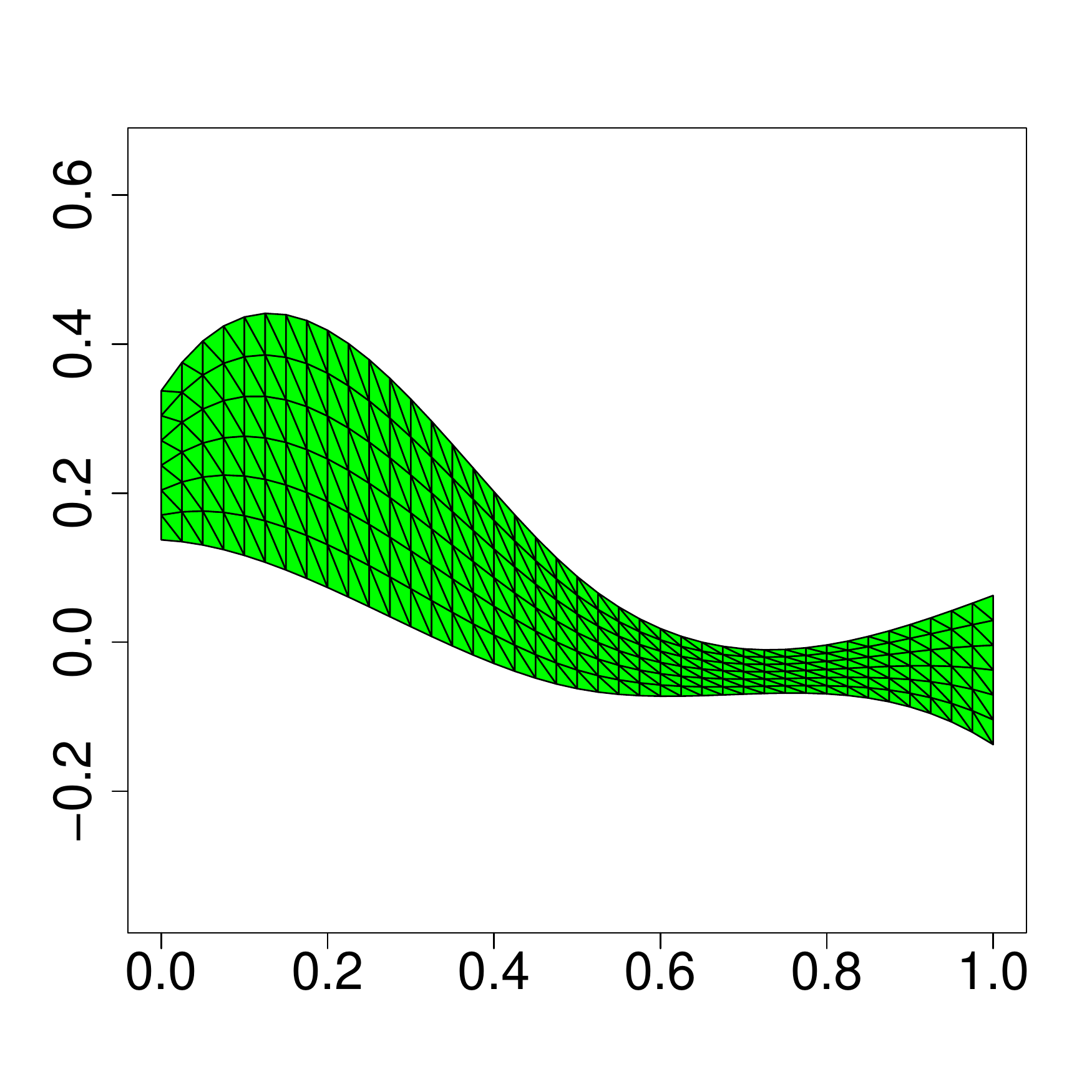}}

 	\end{center}		
	\caption{Exemplary solutions of the weighted sum method of \cite{Doganay2019}, including $x_{0,k',0.25}$ and $x_{0,k'',0.8}$. \label{fig:TC2_WS}}
\end{figure}

 Since this test case is more complex than  Test Case~1 above, 
 we choose two initial values and compare the respective solutions obtained with Pareto tracing. Towards this end, we consider the weighted sum solutions $x_{0,k',0.25}:=x_{0,k'}=x_{k'}(0.25)$ and $x_{0,k'',0.8}:=x_{0,k''}=x_{k''}(0.8)$ as initial values, i.e., the two solutions with the smallest and largest weight for which the gradient descent method from \cite{Doganay2019} converged. Numerical integration is applied on  $[\lambda_l,\lambda_u]=[0.25,0.8]$, moving in positive (forward) direction when starting from $x_{0,k',0.25}$, and moving in negative (backward) direction when starting from $x_{0,k'',0.8}$. In both cases, we use a step length of $h=0.01$. 

\begin{figure}[ht]
	\begin{center}
		\subfloat[Pareto tracing started in $x_{0,k',0.25}$ (purple) and $x_{0,k'',0.8}$ (blue), compared to the weighted sum results from \cite{Doganay2019} (green)\label{fig:TC2_CompareFronts} ]{\includegraphics[width=0.48\textwidth]{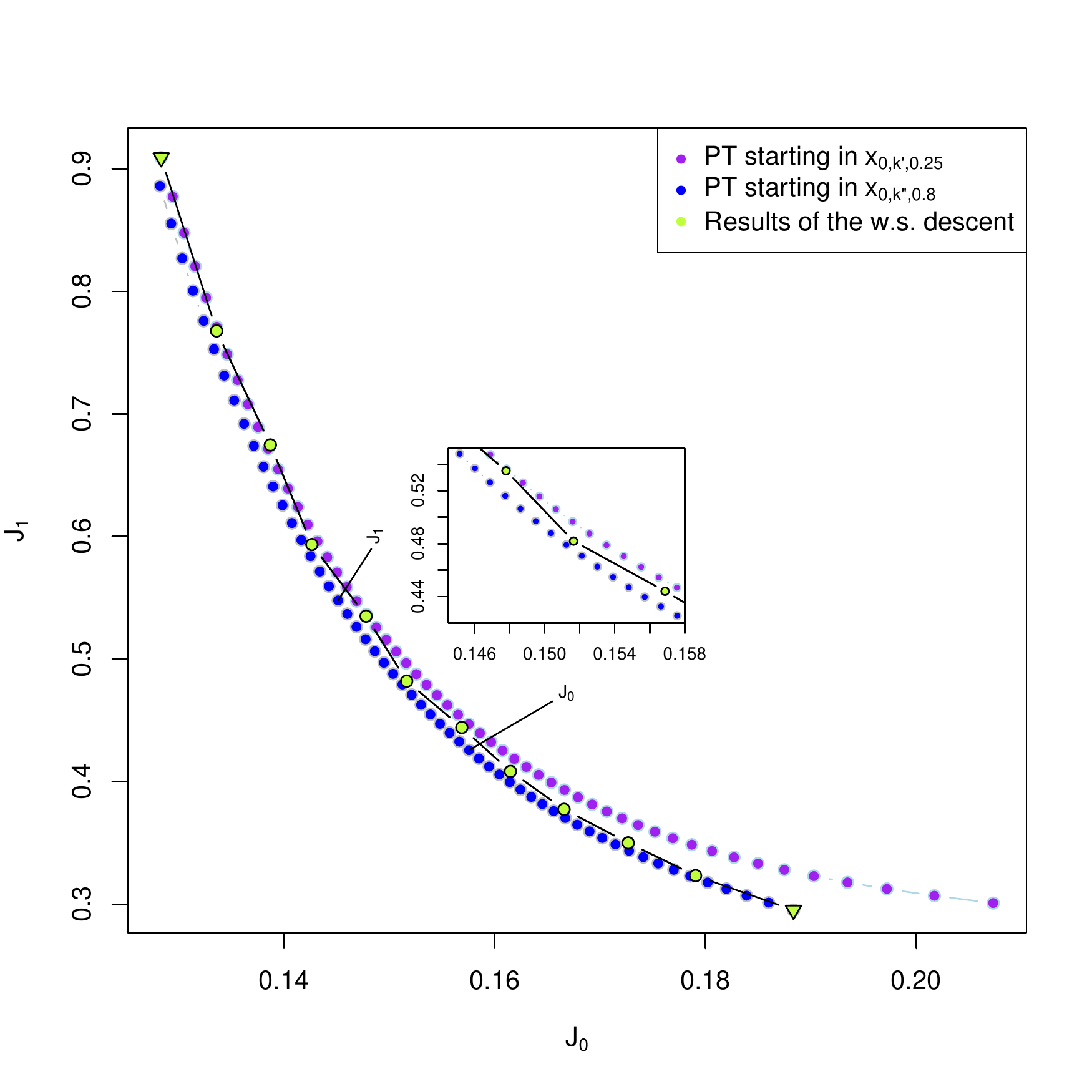}}
		\hspace{0.5cm}
		\subfloat[Pareto tracing  started in premature solutions $x_{0,k_1,0.8}$ (brown), $x_{0,k_2,0.8}$ (green) and $x_{0,k_3,0.8}$ (red), compared to the results for $x_{0,k'',0.8}$ (blue; c.f.\ left figure)  \label{fig:TC2_NOS_PF}
		]{\includegraphics[width=0.48\textwidth]{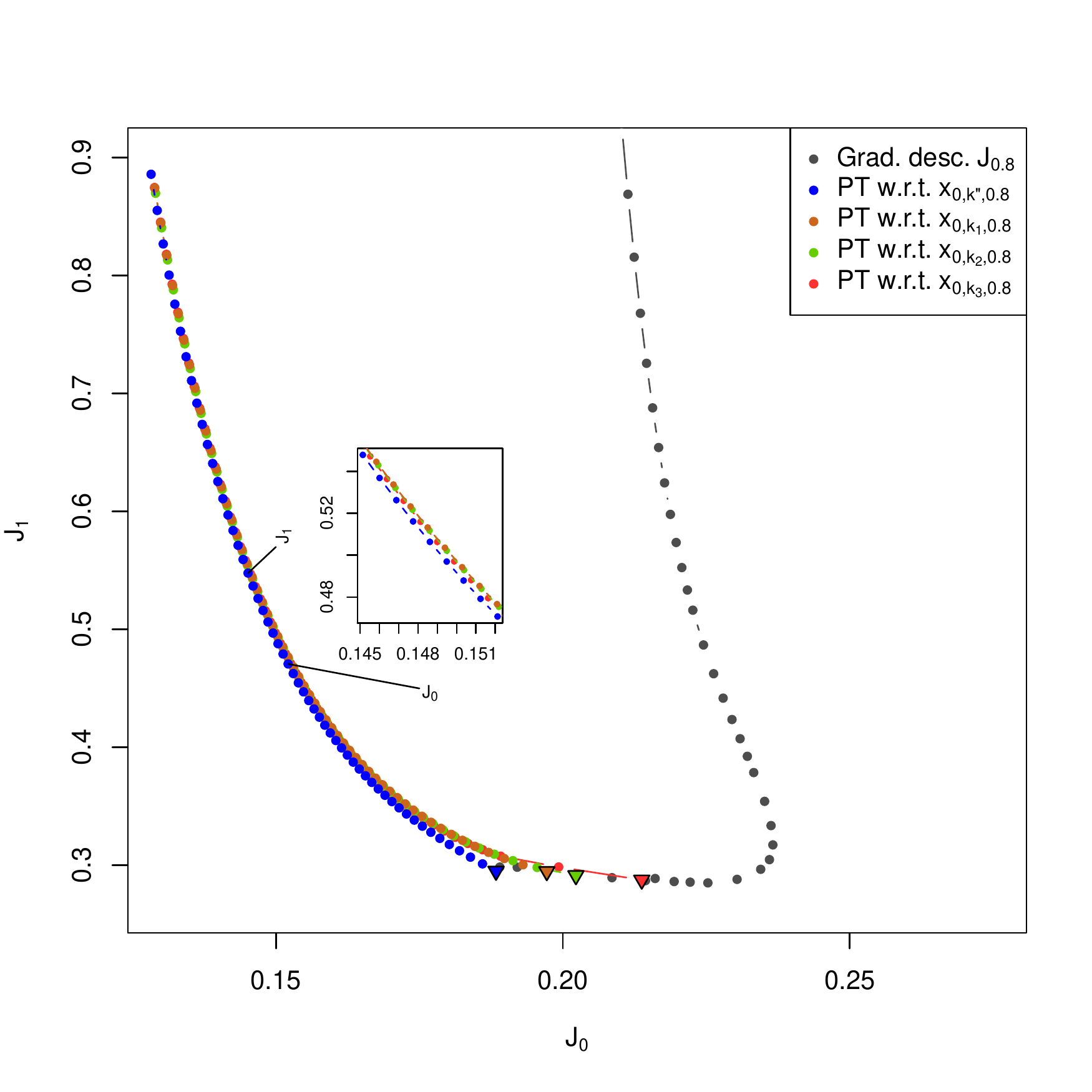}}
 	\end{center}	
	\caption{Comparison of the outcome vectors obtained with Pareto tracing using forward and backward integration (left) and starting from sub-optimal initial solutions (right)}
\end{figure}

A comparison of the outcome vectors obtained from forward and backward Pareto tracing and the results from \cite{Doganay2019} are illustrated in the outcome space in Figure~\ref{fig:TC2_CompareFronts}. The green points correspond to the outcome vectors obtained from the repeated solution of weighted sum scalarizations using gradient descent, where the left most point on the curve corresponds to $x_{0,k',0.25}$ and the right most point corresponds to $x_{0,k'',0.8}$, respectively. Nearly all results for $x_{k'}(\lambda)$ with initial value $\lambda_0=0.25$ (purple trajectory) are dominated by weighted sum solutions, while all of the weighted sum solutions (with obviously the exception of $x_{0,k'',0.8}$) are dominated by the results for $x_{k''}(\lambda)$ with initial value $\lambda_0=0.8$ (blue trajectory).
The shapes obtained for  $x_{k''}(\lambda)$ also resemble the profiles of whales, see Figure~\ref{fig:TC2_ODE_w080}, and are therefore coherent with the weighted sum solutions of \cite{Doganay2019}.

\begin{figure}[ht]

    \centering
	\begin{tikzpicture}[x=\textwidth]
		\node [label=below:{$x_{k''}(0.250)$}] (1) at (0,0) {\includegraphics[width=3cm,keepaspectratio]{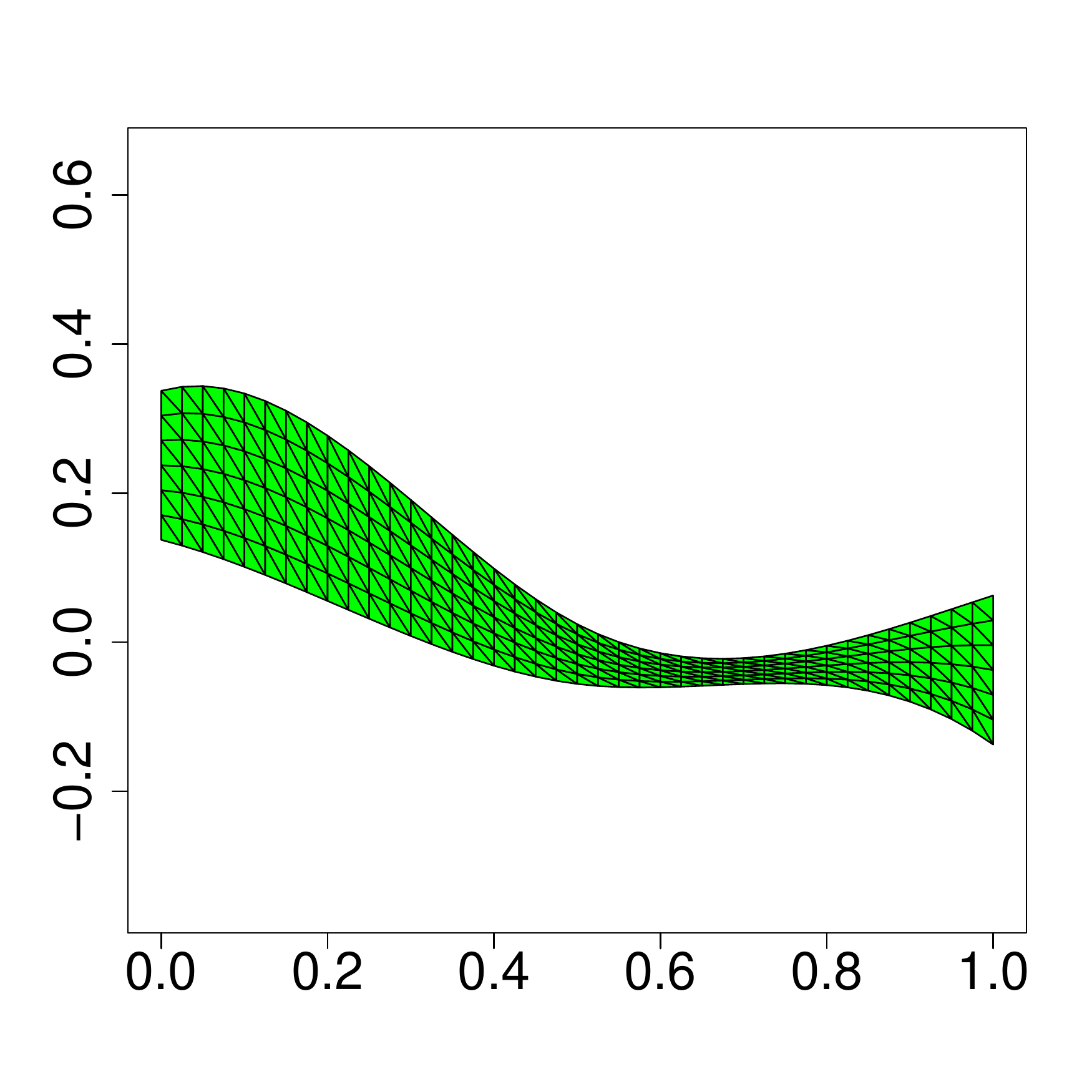}};
		\node [label=below:{$x_{k''}(0.400)$}] (2) at (0.25,0) {\includegraphics[width=3cm,keepaspectratio]{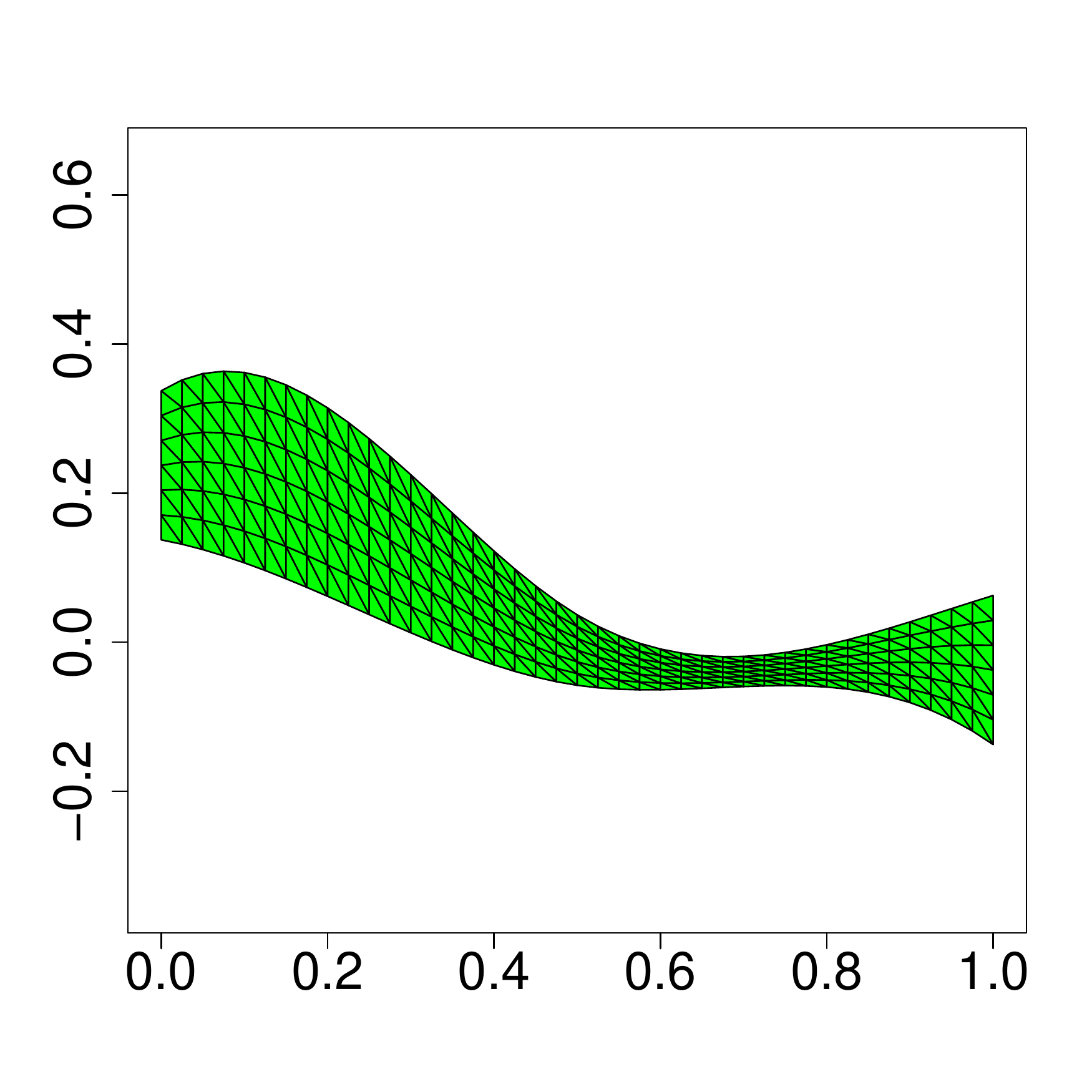}};
		\node [label=below:{$x_{k''}(0.600)$}] (3) at (0.5,0) {\includegraphics[width=3cm,keepaspectratio]{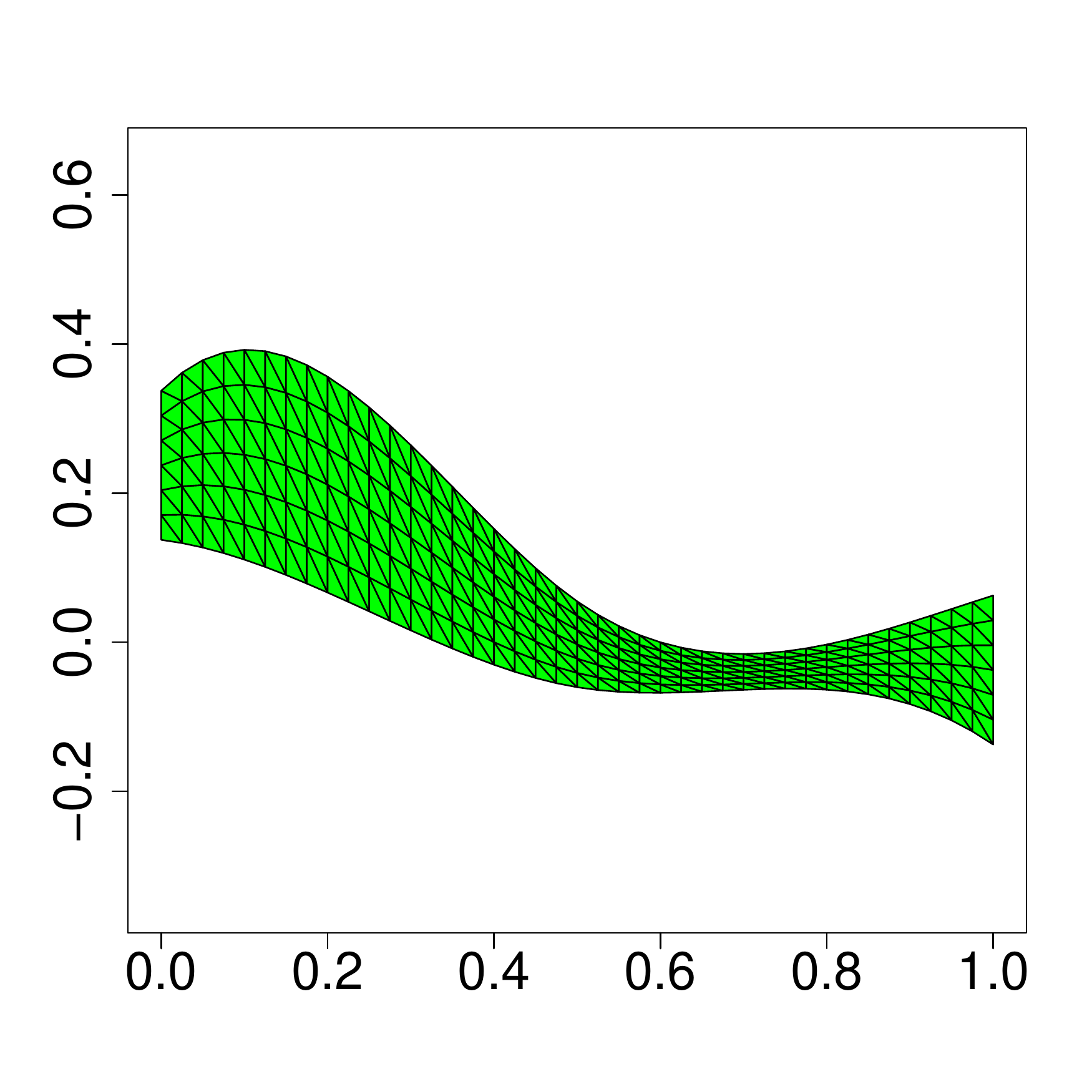}};
		\node [label=below:{$x_{0,k'',0.8}$}] (4) at (0.75,0) {\includegraphics[width=3cm,keepaspectratio]{Figures/ShapeOpt/RM02/RM02_WS_w080.pdf}};

		\draw[->] (2) -- (1) node[midway, above = 0.1cm, sloped] {\tiny{Int. in }} node[midway, below = 0.1cm, sloped] {\tiny{neg. dir.}};
		\draw[->] (3) -- (2) node[midway, above = 0.1cm, sloped] {\tiny{Int. in}} node[midway, below = 0.1cm, sloped] {\tiny{neg. dir.}};
		\draw[->] (4) -- (3) node[midway, above = 0.1cm, sloped] {\tiny{Int. in}} node[midway, below = 0.1cm, sloped] {\tiny{neg. dir.}};
	\end{tikzpicture}	

	\caption{Exemplary shapes obtained with backward Pareto tracing  in $x_{0,k'',0.8}$. \label{fig:TC2_ODE_w080}}
\end{figure}

We further investigated how the solutions differ when the Pareto tracing method is applied starting from a sub-optimal initial value that is obtained if the gradient descent algorithm from \cite{Doganay2019} is stopped prematurely. 
In Figure~\ref{fig:TC2_NOS_PF} the trajectories of three further ODE solves starting in suboptimal initial solutions $x_{0,k_1,0.8}, x_{0,k_2,0.8}$ and $x_{0,k_3,0.8}$, with corresponding initial values $\lambda_{0,k_1}\approx0.808,\lambda_{0,k_2}\approx0.810$ and $\lambda_{0,k_3}\approx0.814$, respectively, are shown. Backward numerical integration with a step length of $h=0.01$ is applied on  $[\lambda_{l,k_i},\lambda_{u,k_i}]=[\lambda_{0,k_i}-0.55,\lambda_{0, k_i}],\ i=1,2,3$, respectively.  
Here, the grey dots show iterates of the gradient descent method applied to the weighted sum objective $J_{0.8}$. 
Despite the relatively bad choices of the initial values, we observe that the solutions w.r.t.\ $k_1,k_2$ and $k_3$ still yield good approximations of the (local) Pareto front, see Figures~\ref{fig:TC2_NOS_PF}  and~\ref{fig:TC2_Optimality}. This can be partially explained by the fact that the gradient descent algorithm applied to the weighted sum objective $J_{0.8}$ first  approaches (an extension of) the Pareto front by making large steps w.r.t.\ $J_1$ (apparently this leads to larger improvements of $J_{0.8}$  in early stages of the optimization process), and moves along the Pareto front during later stages of the optimization when the relation between the potential improvements w.r.t.\ $J_0$ and $J_1$ changes in favor of $J_0$. The sub-optimal initial solutions $x_{0,k_1,0.8}, x_{0,k_2,0.8}$ and $x_{0,k_3,0.8}$ approximate an extension of the Pareto front w.r.t.\ improved $J_1$-values and thus provide very good starting points for Pareto tracing. Note, however, that this is a problem specific observation that largely depends on the value of $\lambda_0$ and, even more so, on the relative variability (slopes) of the considered objective functions. A similar behavior can not be expected in general, as can be seen, for example, in the quadratic case illustrated in Figure~\ref{fig:paretoCQMP}.

\begin{figure}[ht]
	\begin{center}
		\subfloat[First order optimality: $\|\nabla_xJ_{\lambda}(x)\|$ ]{\includegraphics[width=0.5\textwidth]{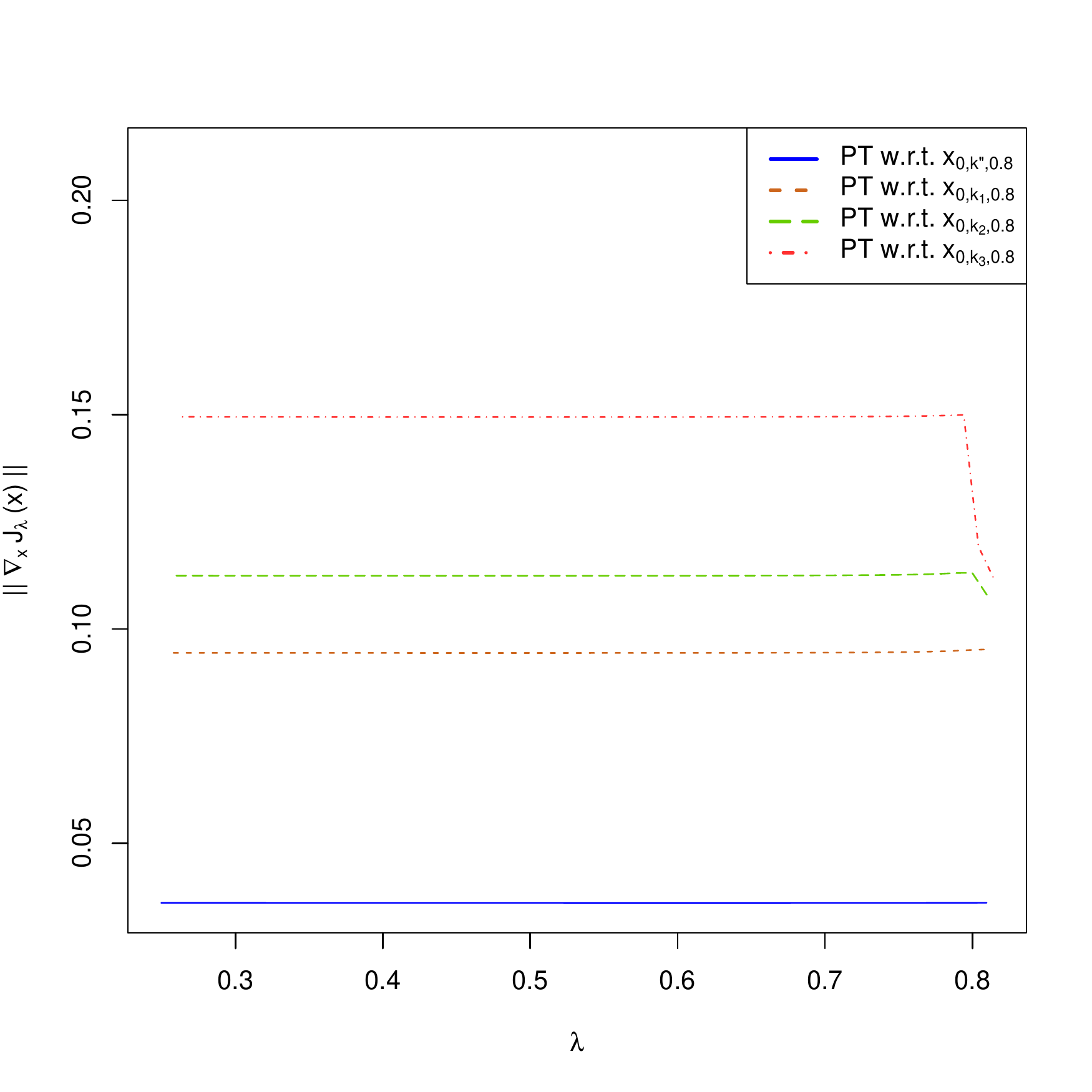}}
		\subfloat[Second order optimality: smallest eigenvalue  $\Lambda(\lambda,x(\lambda))$ of $\nabla_x^2J_{\lambda}(x)$
		]{\includegraphics[width=0.5\textwidth]{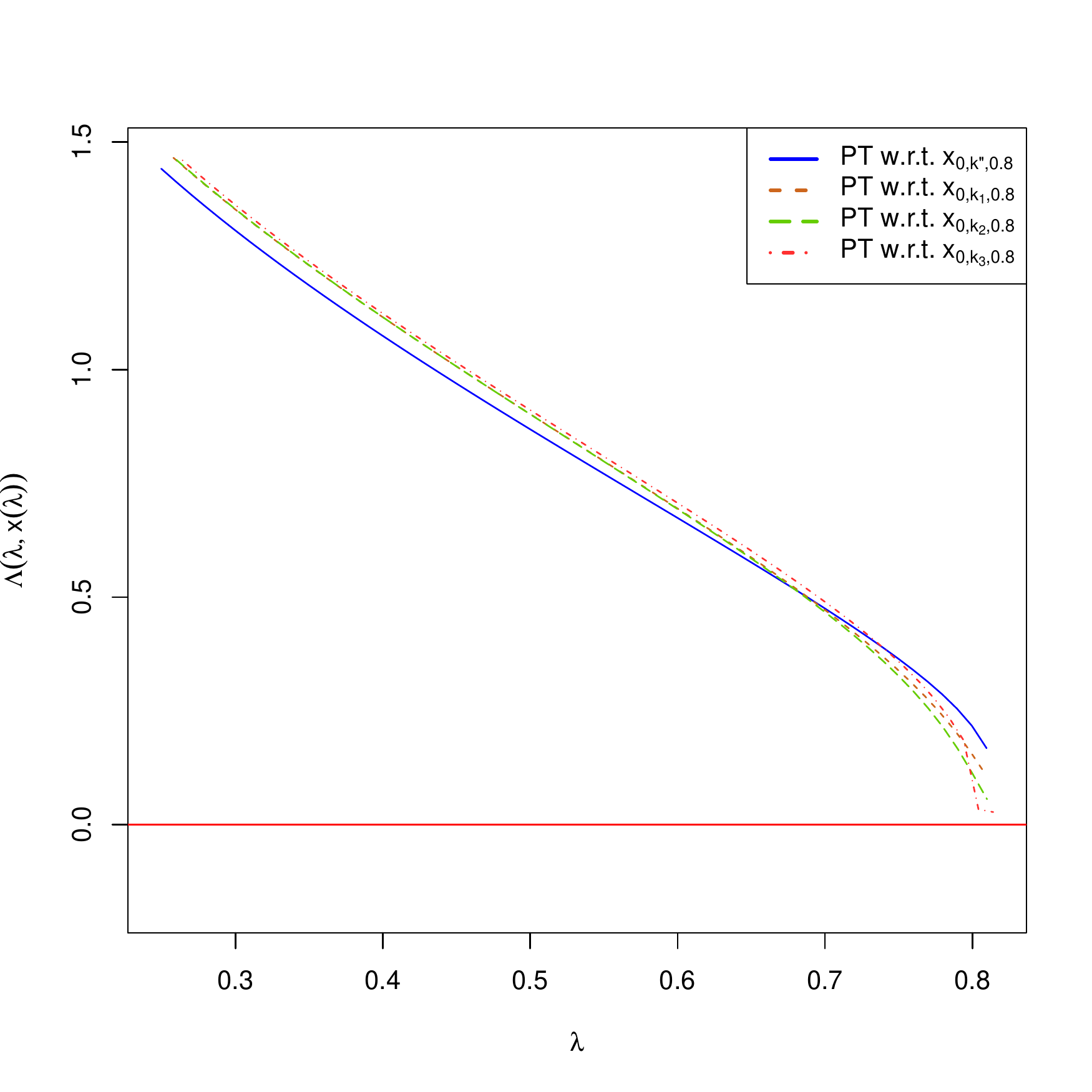}}
 	\end{center}	
	\caption{S-Shaped Joint: Evaluating first and second order optimality during Pareto tracing in dependence of the quality of the initial value\label{fig:TC2_Optimality}	}
\end{figure}

One can also observe that in the above examples Pareto tracing yields a more dense approximation of the (local) Pareto front than the iterative solution of weighted sum problems in \cite{Doganay2019}. Note that this density depends on the choice of the step length $h$ in the Pareto tracing method. Indeed, small step lengths induce dense approximations, however, at comparably high computational costs, while large step lengths may be used to quickly obtain a rough estimate of the Pareto front with rather few and distant solutions. So, the question arises how robust the Pareto tracing approach is w.r.t.\ the step length $h$, and in particular for larger values of $h$. In Figure~\ref{fig:TC2_StepLength} the results of some further ODE solves starting in $x_{0,k'',0.8}$ with different step lengths $h=0.001,0.04,0.08$ are compared. We observe that the results obtained for a larger step length are approximately equal to a subset of the outcome vectors obtained for smaller step lengths (assuming divisibility among the considered step lengths). 
Hence, in this case it is possible to obtain a relatively coarse representation of the (local) Pareto front by using a relatively large step length. This was also observed for the simpler Test Case~1. Note that while the step length $h$ remains constant during the course of the Pareto tracing method, the distance between two consecutive outcome vectors on the approximated (local) Pareto front may differ significantly. This is due to the fact that each iterate $x(\lambda)$ approximates the solution of a weighted sum scalarization $J_{\lambda}$. It is a well-known fact that equally spaced weights $\lambda\in[0,1]$ do in general not yield equally spaced outcome vectors on the Pareto front, see, e.g., \cite{das:aclo:1997} for a detailed analysis of this issue.

From a practical point of view, rough approximations of the Pareto front are of particular interest for computationally expensive problems like the bi-criteria shape optimization problem considered here. Indeed, computing one weighted sum solution with the method suggested in \cite{Doganay2019} came with the cost of $k_W+1$ gradient computations and $k_W\cdot k_A + 1$ objective function evaluations, where $k_W$ denotes the number of iterations of the gradient descent algorithm and $k_A$ denotes the number of Armijo iterations. For Test Case~2 the gradient descent algorithm needed on average $106.7$ iterations, and per iteration on average $5.3$ Armijo iterations to compute a solution for a given weight, i.e., $107.7$ gradient computations and $566.5$ objective function evaluations in total. Given a sufficiently good initial solution, the Pareto tracing approach needs only $14$ gradient computations and one objective function evaluation to compute one further solution. This is a significant speed up that, in combination with the robustness w.r.t.\ the step length, allows for an approximation of a wide range of solutions at reasonable computational cost.
\begin{figure}[h]
    \centering
    \includegraphics[scale=.6]{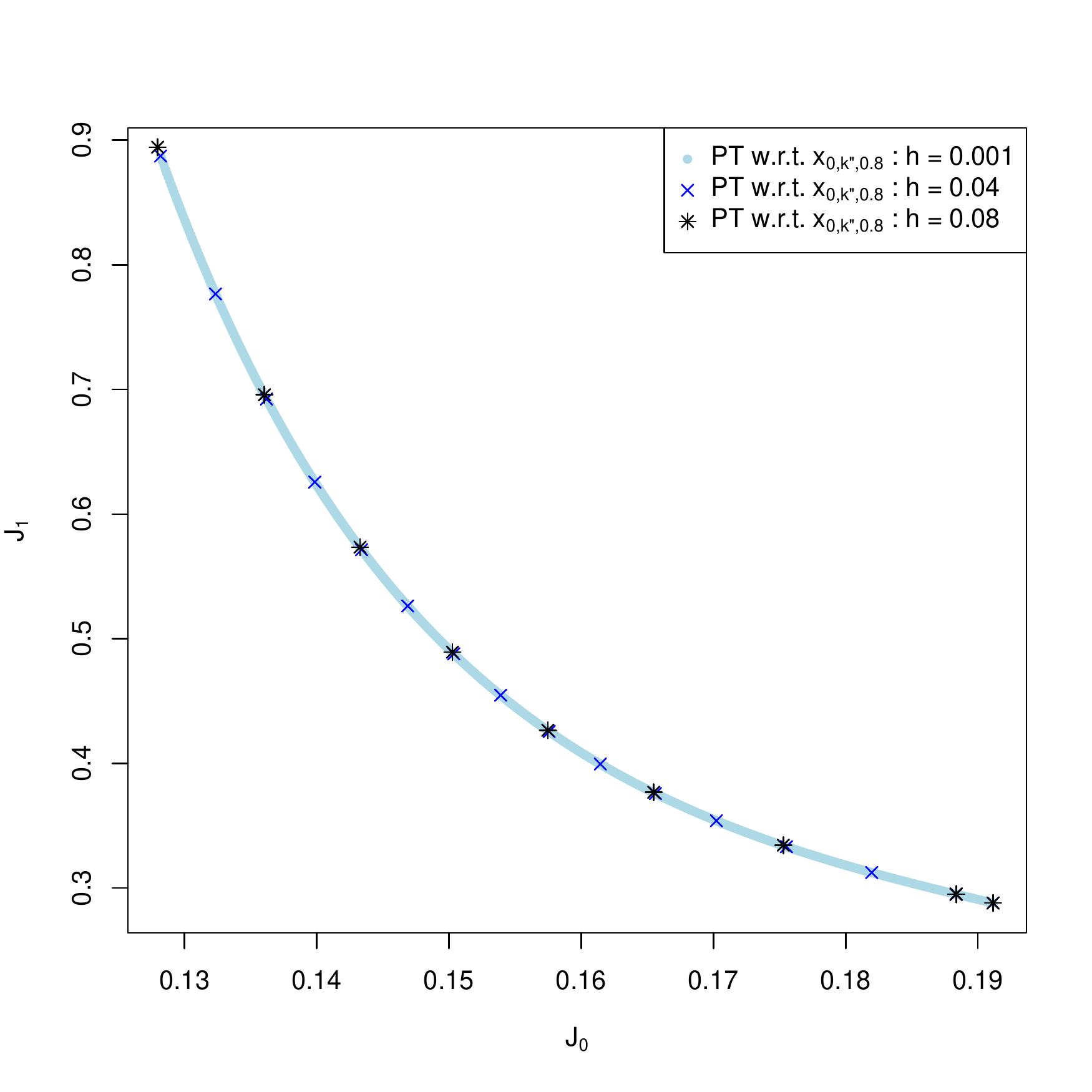}
    \caption{Comparison of the outcome vectors obtained with Pareto tracing started in $x_{0,k'',0.8}$ with different step lengths $h=0.001,0.04,0.08$.}
    \label{fig:TC2_StepLength}
\end{figure}{}

\section{Conclusion and Outlook}

We have presented a novel approach for approximating the Pareto front by tracing it using numerical time integration. The optimality conditions of a scalarization $J_\lambda$ were differentiated w.r.t.\ the scalarization parameter $\lambda$ to obtain an implicit ODE describing the front. If second order optimality conditions are fulfilled, a non-implicit ODE is obtained with a Lipschitz right hand side and the existence and uniqueness of the solution that is a representation of the Pareto front was shown. The smoothness of the Pareto front depends on the smoothness of the objective function. Further, we have shown how this extends to $\epsilon$-critical starting points. The use of standard explicit Runge-Kutta methods was established and the well-known convergence estimates can be applied. The technique was demonstrated for a simple bi-criteria convex quadratic optimization problem, as well as for problems originating from shape optimization.

We have not yet covered the effects of using adapted and/or adaptive step sizes in $\lambda$, e.g., in order to obtain equispaced points on the Pareto front. Different approaches are possible in this respect, see, for example, \cite{Eich09,schm:pare:2008}. Further, we will extend the approach to constrained problems via KKT conditions, and also consider other scalarizations.  While we have only considered the bi-criteria case here, the approach can also be used to handle more than two criteria. In the case of $d+1$ criteria, the front can be described by a $d$-dimensional functional (using again, e.g., weighted sum scalarizations with $d$ independent scalarization parameters) that can be obtained numerically using a $d$-dimensional mesh and numerical integration starting from some mesh point. This will also be considered in the future. 
\vspace{.2cm}

\noindent \textbf{Acknowledgements.}
We thank C. Hahn, M. Reese, J. Schultes, V. Schulz and M. Stiglmayr for interesting discussions and useful hints. M. Bolten, H. Gottschalk and K. Klamroth acknowledge financial support by the Federal Ministry of Education and Research - BMBF through the GIVEN project, grant no.~05M18PXA.

\bibliographystyle{plain}
\bibliography{MOShape}

\end{document}